\documentclass{article}

\usepackage{geometry}
\usepackage{hyperref}

\usepackage{graphicx,wrapfig,lipsum}
\usepackage{color}
\usepackage{transparent}

\usepackage{amsthm}
\usepackage{amssymb}
\usepackage{amsfonts}
\usepackage{amsmath}
\usepackage{mathtools}
\usepackage[shortlabels]{enumitem}

\usepackage{stmaryrd}
\SetSymbolFont{stmry}{bold}{U}{stmry}{m}{n}

\usepackage{stmaryrd}
\usepackage{old-arrows}

\usepackage{leftindex}

\usepackage{mathdots}

\usepackage{listings}

\usepackage{adjustbox}
\usepackage{url} 
\usepackage[T1]{fontenc}

\usepackage{float}

\DeclareFontFamily{U}{mathx}{}
\DeclareFontShape{U}{mathx}{m}{n}{<-> mathx10}{}
\DeclareSymbolFont{mathx}{U}{mathx}{m}{n}
\DeclareMathAccent{\widehat}{0}{mathx}{"70}
\DeclareMathAccent{\widecheck}{0}{mathx}{"71}

\DeclareMathOperator{\Mod}{Mod}

\DeclareMathOperator{\Aut}{Aut}
\DeclareMathOperator{\Hom}{Hom}

\DeclareMathOperator{\sys}{sys}

\DeclareMathOperator{\diam}{diam}
\DeclareMathOperator{\Int}{\iota}
\DeclareMathOperator{\Isom}{Isom}
\DeclareMathOperator{\Bers}{Bers}

\DeclareMathOperator{\Stab}{Stab}
\DeclareMathOperator{\Kob}{Kob}

\DeclareMathOperator{\MO}{\textbf{MO}}

\NewDocumentCommand\vvec{mg}{\overrightarrow{#1}}

\usepackage{tikz} 
\usepackage{tikz-cd}

\usepackage{diagbox}

\usepackage{titlesec}
\titleformat{\subsection}[runin]
  {\normalfont\bfseries}
  {\thesubsection}
  {0.5em}
  {}
  [.]

\theoremstyle{definition}
\newtheorem*{defx*}{Definition}
\newtheorem{definition}{Definition}[section]
\newtheorem{example}[definition]{Example}

\theoremstyle{plain}
\newtheorem{thmx}{Theorem}

\newtheorem{theorem}[definition]{Theorem}
\newtheorem{lemma}[definition]{Lemma}
\newtheorem{corollary}[definition]{Corollary}
\newtheorem{criterion}[definition]{Criterion}
\newtheorem{proposition}[definition]{Proposition}

\theoremstyle{remark}
\newtheorem{remark}[definition]{Remark}

\newcommand{\C}{\mathbb{C}}
\newcommand{\R}{\mathbb{R}}
\newcommand{\Z}{\mathbb{Z}}
\newcommand{\Hyperbolic}{\mathbb{H}}
\newcommand{\Teich}{\mathcal{T}}

\newcommand{\RNum}[1]{\lowercase\expandafter{\romannumeral #1\relax}}
\newcommand{\URNum}[1]{\uppercase\expandafter{\romannumeral #1\relax}}

\DeclareMathOperator{\arcsinh}{arcsinh}

\usepackage[backend=biber,style=alphabetic,sorting=nyt,maxbibnames=10,maxalphanames=10,url=true,doi=false,giveninits=true]{biblatex}
\usepackage{xurl}
\allowdisplaybreaks

\title{Holomorphic Curves in Moduli Spaces\\Are Quasi-Isometrically Immersed}
\author{
  \Large{Yibo Zhang}
  \footnote{
    This work is supported by the French National Research Agency in the framework of the \guillemotleft France 2030 \guillemotright program (ANR-15-IDEX-0002) and by the LabEx PERSYVAL-Lab (ANR-11-LABX-0025-01).
  }\\
\small{\emph{Institut Fourier, UMR 5582, Laboratoire de Mathématiques}}\\
\small{\emph{Université Grenoble Alpes, CS 40700, 38058 Grenoble cedex 9, France}}\\
\small{\emph{email:}} \tt{yibo.zhang@univ-grenoble-alpes.fr}
}

\date{}

\begin{document}


\maketitle

\begin{abstract}

A \emph{holomorphic curve} in moduli spaces is the image of 
a non-constant holomorphic map 
from a cusped hyperbolic surface $B$ of type $(g,n)$ to the moduli space $\mathcal{M}_h$ of closed Riemann surfaces of genus $h$.
We explore the relationship between the holomorphicity and the Teichm\"uller distance.
Our result shows that,
when all peripheral monodromies are of infinite order,
the holomorphic map is a quasi-isometric immersion,
with parameters depending only on $g$, $n$, $h$ and the systole of $B$.
Moreover, under an additional condition on the peripheral monodromies,
the lifting embeds a fundamental domain of the hyperbolic surface $B$ into the Teichm\"uller space as a quasi-isometric embedding.
\end{abstract}


\section{Introduction}

Let $\Teich_h$ be the Teichm\"uller space parameterizing complex structures on a closed, oriented smooth surface $\Sigma_h$ of genus $h\ge 2$, up to isotopy.
This space carries a complex structure induced by an embedding $\Teich_h\hookrightarrow \C^{3h-3}$, due to Bers \cite{Bers1970} and Maskit \cite{Maskit1970}.
The Teichm\"uller distance on $\Teich_h$ will be denoted by $d_{\Teich}$ in the sequel (see Subsection \ref{subsection::Teichmuller} for the definition).

The mapping class group $\Mod_h$ consists of all orientation-preserving diffeomorphisms of $\Sigma_h$ up to isotopy.
This group acts properly discontinuously on $\Teich_h$ and the quotient space is the moduli space $\mathcal{M}_h$.
Each mapping class in $\Mod_h$ is a holomorphic automorphism of $\Teich_h$
and an isometry for the Teichm\"uller distance.

A \emph{holomorphic disc} in $\Teich_h$ is the image of a holomorphic map $\widetilde{F}:\Hyperbolic^2\rightarrow \Teich_h$ from the upper half plane $\Hyperbolic^2\subset \C$ to the Teichm\"uller space.
The hyperbolic plane is endowed with the usual complex structure making it biholomorphic to the open unit disc.
Passing to the quotient we obtain a map $\mathring{F}:\Hyperbolic^2\rightarrow\mathcal{M}_h$ to the moduli space.
Let $\Gamma\le \Stab(\mathring{F})\coloneqq \{\phi\in \Aut(\Hyperbolic^2) \mid \mathring{F} \circ \phi=\mathring{F} \}$ be a lattice and set $B\coloneqq \Gamma \backslash \Hyperbolic^2$.
Suppose that $B$ is an oriented surface of genus $g$ with $n$ cusps (without boundary).
The image of the quotient map $F:B\rightarrow\mathcal{M}_h$ is called a \emph{holomorphic curve} of type $(g,n)$ in $\mathcal{M}_h$.

Farb and Masur \cite{Farb-Masur-TeichmullerGeom1,Farb-Masur-TeichmullerGeom2} established a connection between the coarse geometric structure and the analytic structure of the moduli space. In this paper, we investigate the holomorphicity of the immersed curve $F(B) \subset \mathcal{M}_h$ in relation to the Teichm\"uller distance.

We denote the hyperbolic distance on $B$ by $d_B$. 
The \emph{systole} of $B$, denoted by $\sys(B)$, is the length of the shortest essential closed curve, where \emph{essential} means non-contractible and non-peripheral.
A \emph{cusp region} of $B$, usually denoted by $U$, is a neighbourhood of a cusp bounded by a horocycle of length $2$.

As in \cite[p.212]{ImayoshiShiga1988} and \cite[p.176]{moishezon1977complex}, a holomorphic map $F:B\rightarrow\mathcal{M}_h$ induces a group homomorphism $F_*:\pi_1(B,t) \rightarrow \Mod_h$ which is called a \emph{monodromy homomorphism} of $F$ (see Subsection \ref{subsection::monodromy} for the definition).
The image $F_*([\gamma])$ along a based loop $\gamma\subset B$ that goes once or several times around a cusp, clockwise or counterclockwise, is called a \emph{peripheral monodromy} of the cusp.

Both $\Hyperbolic^2$ and $\Teich_h$ are complex manifolds and Kobayashi hyperbolic, i.e., 
they are endowed with the intrinsic \emph{Kobayashi norms}.
The distance induced by the Kobayashi norms is called the \emph{Kobayashi distance}.
The Kobayashi norm on $\Hyperbolic^2$ coincides with the norm of the Poincar\'e Riemannian metric or half the hyperbolic metric.
The Kobayashi distance on $\Hyperbolic^2$ coincides with half the hyperbolic distance, i.e.,
$d_{\Hyperbolic^2,\Kob} = \frac{1}{2}d_{\Hyperbolic^2}$ (see, e.g., \cite[Proposition 2.3.4]{Abate}).
On the other hand, the Kobayashi norm $\Kob_{\Teich}$ on $\Teich_h$ is a Finsler metric and the Kobayashi distance on $\Teich_h$ coincides with the Teichm\"uller distance, i.e., $d_{\Teich,\Kob} = d_{\Teich}$ (see \cite[Theorem 3]{Royden1971}).

Consider the distance $d_{\widetilde{F}}$ on $\Hyperbolic^2$ induced by the pullback norm $\widetilde{F}^* \Kob_{\Teich}$.
The quotient metric of $d_{\widetilde{F}}$ on $B$ is denoted by $d_F$,
which measures the distance between $b_1$ and $b_2$
as the infimum length of paths on the image surface $F(B)$ connecting $F(b_1)$ and $F(b_2)$.
The quotient Teichm\"uller distance on $\mathcal{M}_h$ is denoted by $d_{\mathcal{M}}$.
Here, the distance $d_\mathcal{M}(F(b_1),F(b_2))$ measures the shortest path in $\mathcal{M}_h$ joining $F(b_1)$ and $F(b_2)$, which, in general, does not lie on the image surface $F(B)$.
The holomorphic map $F:B\rightarrow\mathcal{M}_h$
is an \emph{isometric immersion}
if the Finsler norm on the left equals the pullback norm or, equivalently, the equality $(1/2)d_B = d_F$ holds.

To analyse the rigidity of
the maps \( F: B \to \mathcal{M}_h \)
and
\( \widetilde{F}: \mathbb{H}^2 \to \mathcal{T}_h \),
we first recall the classical definition of a \emph{quasi-isometric embedding}.
Furthermore, we introduce the notion of a \emph{quasi-isometric immersion},
which is a quasi-version of the relation $(1/2)d_B = d_F$,
which provides a framework for examining the rigidity of \( F: B \to \mathcal{M}_h \).

\begin{definition}
Let $(X_1,d_1)$ and $(X_2,d_2)$ be metric spaces.
Given $\lambda\ge 1$ and $\epsilon\ge 0$,
a map $f:X_1\rightarrow X_2$ is called a $(\lambda,\epsilon)$-\emph{quasi-isometric embedding} if
\begin{equation*}
    d_1(x,y)/\lambda-\epsilon
    \le
    d_2(f(x),f(y))
    \le
    \lambda d_1(x,y)+\epsilon
\end{equation*}
for all $x,y\in X_1$.
In particular, when $f$ is a $(1,0)$-quasi-isometric embedding, we say that $f$ is an \emph{isometric embedding}.
\end{definition}

\begin{definition}
Given $\lambda\ge 1$ and $\epsilon\ge 0$,
a holomorphic map $F:B\rightarrow\mathcal{M}_h$ is called a $(\lambda,\epsilon)$-\emph{quasi-isometric immersion} if
\begin{equation*}
(1/2)d_B(b_1,b_2)/\lambda-\epsilon
\le
d_F(b_1,b_2)
\le
(\lambda/2) d_B(b_1,b_2)+\epsilon
\end{equation*}
for all $b_1,b_2\in B$.
In this case, we say that $F(B)$ is \emph{quasi-isometrically immersed}.
In particular, when $F$ is a $(1,0)$-quasi-isometric immersion, we say that $F$ is an \emph{isometric immersion} and $F(B)$ is \emph{isometrically immersed}.
\end{definition}

\begin{remark}
Just as a quasi-isometric embedding needs not be an embedding, a quasi-isometric immersion needs not be an immersion in the usual sense.
However, an isometric immersion
is indeed an immersion.
\end{remark}

A holomorphic map $f:X_1\rightarrow X_2$ between complex manifolds is distance-decreasing for the intrinsic Kobayashi distances, namely $d_{X_1,\Kob}(x,y)\ge d_{X_2,\Kob}(f(x),f(y))$ (see, e.g., \cite[Proposition 2.3.1]{Abate}).
In particular, the holomorphic map $F:B \rightarrow \mathcal{M}_h$ and its lift $\widetilde{F}:\Hyperbolic^2 \rightarrow \Teich_h$ satisfy the following inequalities:
\begin{align*}
\frac{1}{2} d_B\Big( b_1,b_2 \Big)
\ge
d_{F}\Big( b_1,b_2 \Big)
\ge
d_{\mathcal{M}}\Big( F(b_1),F(b_2) \Big),\\
\frac{1}{2} d_{\Hyperbolic^2}\Big( \widetilde{b_1},\widetilde{b_2} \Big)
\ge
d_{\widetilde{F}}\Big( \widetilde{b_1},\widetilde{b_2} \Big)
\ge
d_{\Teich}\Big( \widetilde{F}(\widetilde{b_1}),\widetilde{F}(\widetilde{b_2}) \Big).
\end{align*}
Consequently, each peripheral monodromy must be either reducible or of finite order (see Corollary \ref{corollary::peripheral monodromy}).
In particular, if a peripheral monodromy $\phi$ is of infinite order, then a power $\phi^{\mu}$ is a multi-twist.

\begin{remark}
If a holomorphic map $F:(B,(1/2) d_B) \rightarrow (\mathcal{M}_h, d_\mathcal{M})$ is a $(\lambda,\epsilon)$-quasi-isometric embedding, then it must also be a $(\lambda,\epsilon)$-quasi-isometric immersion.
However, the converse does not hold in general.
\end{remark}

An isometrically immersed curve $F(B)\subset \mathcal{M}_h$ is known as a \emph{Teichm\"uller curve}.
The first non-trivial cases were discovered by Veech (\cite{Veech1989}) and all Teichm\"uller curves in $\mathcal{M}_2$ and $\mathcal{M}_3$ are almost well-understood (\cite{McMullen2005genus2,McMullen2006genus2} and \cite[][Theorem 5.5]{Lanneau-Nguyen2014,McMullen2023}).
In $\mathcal{M}_h$ with $h\ge 5$, Teichm\"uller curves are elusive and the only known primitive case was given in \cite{Bouw-Moller2010}.

\subsection*{Quasi-isometric rigidity for holomorphic curves}

Every isometric immersion from $B$ to $\mathcal{M}_h$ is either holomorphic or anti-holomorphic (see \cite{Antonakoudis2017}).
However, a holomorphic map needs not be an isometric immersion.
The following theorem shows that a weaker statement still holds:

\begin{thmx}
\label{thmx::quasi-isometry}
Let \((g, n)\), \(h\) and \(\epsilon\) be given such that \(2g - 2 + n > 0\), \(h \geq 2\) and \(\epsilon > 0\).
There exists a constant \(K = K(g, n, h, \epsilon)\), depending only on \(g\), \(n\), \(h\) and \(\epsilon\), with the following property:
Let \(B\) be an oriented hyperbolic surface of type \((g, n)\) with systole \(\sys(B) \geq \epsilon\) and cusp regions \(U_1, \dots, U_n \subset B\).
Let \(F: B \rightarrow \mathcal{M}_h\) be a non-constant holomorphic map with monodromy homomorphism \(F_* \in \Hom(\pi_1(B, t), \Mod_h)\).
Then:

\begin{enumerate}
\item \textbf{Quasi-isometric embedding on cusp regions:} For each \(i = 1, \dots, n\), the restriction  
\[
   \left.F\right|_{U_i} : (U_i, (1/2)d_B) \rightarrow (\mathcal{M}_h, d_{\mathcal{M}})
\]  
is a quasi-isometric embedding if and only if the corresponding peripheral monodromy has infinite order.
In this case, the restriction $\left.F\right|_{U_i}$ is a  \((1, K)\)-quasi-isometric embedding.

\item \textbf{Quasi-isometric immersion globally:} The map \(F\) is
a quasi-isometric immersion if and only if
all peripheral monodromies are of infinite order.
In this case, the map $F$ is a \((1, K)\)-quasi-isometric immersion.
\end{enumerate}
\end{thmx}

The condition that peripheral monodromies must be of infinite order in Theorem \ref{thmx::quasi-isometry} is ``obviously necessary''.
This necessity follows from the theory of the Deligne-Mumford compactification
of $\mathcal{M}_h$, 
particularly from stable reduction theory.
For further details, we refer to Proposition \ref{proposition::torsion-peripheral}.

Example \ref{example::non-quasi-cusp} shows
the case where the peripheral monodromy at the $i$-th cusp has finite order, yet $F(U_i)$ is located in the thick part of $\mathcal{M}_h$.
A sample holomorphic curve demonstrating this behaviour is depicted in Figure \ref{figure::holomorphic-curve}.

Moreover, Theorem \ref{thmx::quasi-isometry} is optimal:
there exist holomorphic curves in $\mathcal{M}_h$ that are quasi-isometrically but not isometrically immersed, as shown in Example \ref{example::quasi-example}.

\subsection*{Quasi-isometric rigidity for fundamental domains}

If the holomorphic map $F:B\rightarrow \mathcal{M}_h$ is an isometric immersion,
then the lift $\widetilde{F}:\Hyperbolic^2\rightarrow\Teich_h$ is
a complex geodesic for the intrinsic Kobayashi norms.
Teichm\"uller's uniqueness theorem (see, e.g., \cite[Theorems 11.8 and 11.9]{FarbMargalit2011}) shows that any two points of the Teichm\"uller space $\Teich_h$ are joined by a unique real geodesic.
Therefore, the lift $\widetilde{F}: (\Hyperbolic^2, (1/2)d_{\Hyperbolic^2}) \rightarrow (\Teich_h,d_{\Teich})$ is an isometric embedding.

If the holomorphic map $F:B\rightarrow \mathcal{M}_h$ is a quasi-isometric immersion,
in general, the lift $\widetilde{F}: (\Hyperbolic^2, (1/2)d_{\Hyperbolic^2}) \rightarrow (\Teich_h,d_{\Teich})$
fails to be a quasi-isometric embedding.

We now aim to obtain a hyperbolic polygon $D\subset \Hyperbolic^2$, i.e. a fundamental domain of $B$ bounded by geodesic segments, such that $\left.\widetilde{F}\right|_{D}: (D,(1/2)d_{\Hyperbolic^2}) \rightarrow (\Teich_h,d_{\Teich})$ is a quasi-isometric embedding.
To simplify the statement, we consider only the case $(g,n)=(0,n)$
and leave the statement for the general case in Subsection \ref{subsection::Theorem B}.

Given a holomorphic map $F:B\rightarrow\mathcal{M}_h$ and a monodromy homomorphism $F_*:\pi_1(B,t)\rightarrow \Mod_h$,
as $(g,n)=(0,n)$, one can describe the monodromy homomorphism $F_*:\pi_1(B,t)\rightarrow \Mod_h$ as an $n$-tuple in $\Mod_h$ as follows.
Choose homotopy classes of loops $\gamma_1, \ldots, \gamma_n \subset B$ based at $t$ such that $\gamma_i$ goes around the $i$-th cusp exactly once clockwise and $\gamma_n\cdots\gamma_1$ is homotopically trivial, for $i=1,\ldots,n$.
The $n$-tuple $(\phi_1,\ldots,\phi_n)\coloneqq (F_*(\gamma_1),\ldots,F_*(\gamma_n))$ is called a \emph{global monodromy} of $F$.
Note that $\phi_1\cdots \phi_n=1$.

The following condition on the monodromy is sufficient:

\begin{thmx}
\label{thmx::quasi-isometry-fundamental-domain-n=0}
Given $n\ge 3$, $h\ge 2$ and $\epsilon>0$,
there exists a constant $K=K(0,n,h,\epsilon)$, depending only on $n,h$ and $\epsilon$, with the following property:
Let $B$ be an oriented hyperbolic surface of type $(0,n)$ with $\sys(B)\ge \epsilon$.
Let $F:B\rightarrow\mathcal{M}_h$ be a non-constant holomorphic map.
Suppose that $(\phi_1,\ldots,\phi_n)$ is a global monodromy of $F$ satisfying:
\begin{itemize}[-]
\item each $\phi_i$ is of infinite order
and thus $\phi_i^{\mu_i} = T_{\boldsymbol{a}_i}$ is a multi-twist for some $\mu_i\ge 1$,

\item $\boldsymbol{a}_1,\ldots,\boldsymbol{a}_n$ are pairwise intersecting.
\end{itemize}
Then, there exists a fundamental polygon $D$ of $B$ such that
$\left.\widetilde{F}\right|_{D}: (D,(1/2)d_{\Hyperbolic^2}) \rightarrow (\Teich_h,d_{\Teich})$ is a $(2,K+\diam(D_{\text{cp}}))$-quasi-isometric embedding,
where $D_{\text{cp}}$ denotes the compact part of $D$ obtained by removing the lift of all cusp regions.
\end{thmx}

The second hypothesis of Theorem \ref{thmx::quasi-isometry-fundamental-domain-n=0} is a mild condition.
In particular, a holomorphic curve of type $(0,n), n\ge 3$ in $\mathcal{M}_2$ with each peripheral monodromy the Dehn twist along a non-separating closed curve (i.e. the image of the classifying map of a genus-$2$ Lefschetz fibration without reducible fibres) must have a quasi-isometrically embedded fundamental polygon (see Subsection \ref{subsection::genus2Lef}).

\subsection*{Outline}

Section \ref{section::Preliminary} introduces the preliminaries for studying the Teichmüller space and the monodromy of a holomorphic map
$F:B\rightarrow \mathcal{M}_h$.
Subsection \ref{subsection::eppA} establishes a strong monodromy condition, characterizing Teichmüller curves and proving the optimality of Theorem \ref{thmx::quasi-isometry}.

Subsection \ref{subsection::irreducibility} uses the irreducibility of holomorphic fibrations to prove Theorem \ref{theorem::Ct_belongs_to_compactK}, which states that every non-constant holomorphic curve intersects a thick part of the moduli space.
In Subsection \ref{subsection::cusp}, we analyse holomorphic hyperbolic cusp regions in $\mathcal{M}_h$.
A key observation is that an infinite-order mapping class, acting as a multi-twist along small geodesics, perturbs a marked hyperbolic surface only slightly in $\Teich_h$.
This reveals a ``force'' from the cusp, partially establishing the quasi-isometric embedding in Theorem \ref{thmx::quasi-isometry} - (\RNum{1}).
Theorem \ref{theorem::Ct_belongs_to_compactK} shows how the thick part of the moduli space exerts a ``force'' on holomorphic curves and cusp regions, completing the proof of Theorem \ref{thmx::quasi-isometry} in Subsection \ref{subsection::Theorem A}. Subsection \ref{subsection::Theorem B} proves Theorem \ref{thmx::quasi-isometry-fundamental-domain-n=0} and discusses quasi-isometrically embedded fundamental domains in the general case.

Finally, Section \ref{section::example} provides examples and an application to genus-$2$ Lefschetz fibrations.

\subsection*{Acknowledgements}
This work is part of my PhD thesis.
I would like to thank my advisors, Louis Funar and Greg McShane, for their helpful discussions and guidance. I would also like to thank Stepan Orevkov for useful suggestions.
I am deeply grateful to the referees for their thoughtful feedback and constructive comments, which helped refine and strengthen the paper.
\section{Preliminary}
\label{section::Preliminary}

We fix non-negative integers $g,n$ and $h$ with $2g-2+n>0$ and $h\ge 2$ once and for all.

\subsection{Teichm\"uller space}
\label{subsection::Teichmuller}

Let $\Sigma_{g,n}$ be an oriented smooth surface of genus $g$ with $n$ punctures without boundary.
The \emph{geometric intersection number} of two closed curves $\gamma_1$ and $\gamma_2$, denoted by $\Int(\gamma_1,\gamma_2)$, is the minimum cardinality of $\nu_1\cap \nu_2$ for all closed curves $\nu_1$, $\nu_2$ such that $\gamma_i$ is homotopic to $\nu_i$, for $i=1,2$.

We say that a set of closed curves $\{\gamma_1,\ldots,\gamma_k\}$ \emph{fills up} the surface $\Sigma_{g,n}$ if, for any non-contractible non-peripherial closed curve $\nu$, $\Int(\nu,\gamma_i)\ge 1$ for some $i$.
A set of disjoint simple closed curves on $\Sigma_{g,n}$ is called a multi-curve.
By convention, we define the product $\gamma_1\cdot \gamma_2$ of two oriented paths as their concatenation and the inverse $\gamma^{-1}$ of an oriented path is the same path with the opposite orientation.

The Teichm\"uller space $\Teich_{g,n}$ consists of all marked Riemann surfaces of type $(g,n)$, i.e. equivalent pairs $(X,f_X)$ where
$X$ is a Riemann surface and 
$f_X:\Sigma_{g,n}\rightarrow X$ is an orientation preserving diffeomorphism.
Two pairs $(X,f_X)$, $(Y,f_Y)$ are equivalent if $f_Y\circ f_X^{-1}:X\rightarrow Y$ is isotopic to a biholomorphism.
Based on the uniformization of Riemann surfaces, each equivalent class is represented by a marked hyperbolic surface.
Then a mapping class $[\phi]\in\Mod_{g,n}$ acts on $\Teich_{g,n}$ by
$[\phi]\cdot [(X,f_X)]=[(X,f_X\circ \phi^{-1})]$.
The mapping class group acts properly discontinuously on the Teichm\"uller space and the quotient space is called the moduli space, denoted by $\mathcal{M}_{g,n}$.

For any two points $[(X,f_X)]$, $[(Y,f_Y)]\in\Teich_{g,n}$, we define the \emph{Teichm\"uller distance} $d_{\Teich}$ by
\begin{equation}
\label{equation::Teichmuller distance}
    d_{\Teich}([(X,f_X)],[(Y,f_Y)])=\frac{1}{2} \log \inf_{\phi}\{K(\phi)\}.
\end{equation}
Here, the infimum is taken over all quasiconformal homeomorphisms $\phi:X\rightarrow Y$ homotopic to $f_Y\circ f_X^{-1}$, i.e. all quasiconformal homeomorphisms respecting the markings.
Moreover, $K(\phi)\ge 1$ denotes dilatation of $\phi$.

The \emph{geodesic length function}
$L_{\gamma}([(X,f_X)])$
assigns to each closed curve $\gamma\subset \Sigma_{g,n}$ the length of the unique geodesic homotopic to $f_X(\gamma)$ on the hyperbolic representative $X$ of $[(X,f_X)] \in \Teich_{g,n}$.
The length of a closed curve $\gamma$ on a hyperbolic surface $X$ is denoted by $l_X(\gamma)$.
The next lemma is due to Wolpert.

\begin{lemma}[\cite{Wolpert1979}]
\label{lemma::wolpert}
Consider two points in $\Teich_{g,n}$ that are represented by the marked hyperbolic surfaces $(X,f_X)$ and $(Y,f_Y)$.
Set $K=\exp{\big( 2d_{\Teich}([(X,f_X)],[(Y,f_Y)]) \big) }$.
Then, the geodesic length function is distorted by a factor of at most $K$, i.e.
\begin{equation*}
    \frac{1}{K}L_{\gamma}([(X,f_X)])\le L_{\gamma}([(Y,f_Y)])\le K L_{\gamma}([(X,f_X)])
\end{equation*}
for any closed curve $\gamma\subset\Sigma_{g,n}$.
\end{lemma}

Recall the classification of mapping classes due to Bers (\cite{Bers1978}).
Let $\varphi\in\Mod_h$ be a mapping class.
The \emph{translation distance} of $\varphi$ is defined by
\begin{equation*}
\tau(\varphi)=
\inf_{[(X,f_X)]\in\Teich_h}d_{\Teich}([(X,f_X)],\varphi\cdot [(X,f_X)]).
\end{equation*}
Then $\tau(\varphi)=0$ and the infimum is attained if and only if $\varphi$ is periodic.
Also $\tau(\varphi)$ is positive and the infimum is attained if and only if $\varphi$ is pseudo-Anosov.
Eventually, $\tau(\varphi)$ is not attained if and only if $\varphi$ is reducible and of infinite order.
Furthermore, if $\tau(\varphi)=0$ and $\tau(\varphi)$ is not attained, then
then there exists $\mu \in \Z_{\ge 1}$ bounded above by a constant determined by $h$ such that $\phi^{\mu}$ is a multi-twist.
\subsection{Monodromy homomorphisms}
\label{subsection::monodromy}

Suppose that $B=\Gamma \backslash \Hyperbolic^2$ is a hyperbolic surface of type $(g,n)$ with some lattice $\Gamma\le \Aut(\Hyperbolic^2)$.
Consider a holomorphic map $F:B\rightarrow\mathcal{M}_h$ and the lift $\widetilde{F}:\Hyperbolic^2\rightarrow \Teich_h$.
We introduce the induced \emph{monodromy homomorphisms}.

Firstly, the map $\widetilde{F}:\Hyperbolic^2\rightarrow\Teich_h$ induces a group homomorphism $F_{\Gamma}:\Gamma\rightarrow \Mod_h$ such that
$\widetilde{F}\circ \phi = F_{\Gamma}(\phi)\circ \widetilde{F}$,
for every $\phi\in\Gamma$.
When $F(t)$ has non-identity automorphisms (i.e. $F(t)$ is symmetric) for some $t\in B$, the homomorphism $F_{\Gamma}$ is not necessarily unique.
There are at most $(2g+n)^{84(h-1)}$ many possibilities of such a homomorphism $F_{\Gamma}$.

Secondly, fixing a base point $t\in B$ and lifting it to some $\widetilde{t}\in\Hyperbolic^2$, we obtain a standard group isomorphism $\rho_{t,\widetilde{t}}:\pi_1(B,t)\rightarrow \Gamma$ as follows.
A loop $\gamma\subset B$ based at $t$ is lifted to the path in $\Hyperbolic^2$ joining $\widetilde{t}$ and $\rho_{t,\widetilde{t}}([\gamma]) \cdot \widetilde{t}$.

Eventually,
$F_*\coloneqq F_{\Gamma}\circ \rho_{t,\widetilde{t}} \in \Hom(\pi_1(B,t), \Mod_h)$
is called a \emph{monodromy homomorphism} of $F$.
A different choice of $\widetilde{t}$ may change the monodromy homomorphism by a conjugacy.
The monodromy homomorphism enables us to reformulate the most important property of the holomorphic map $F$, say being distance-decreasing for the intrinsic Kobayashi distances, as follows.

\begin{proposition}
\label{proposition::distance-decreasing-1}
Let $\gamma\subset B$ be a loop based at $t$.
Then $(1/2) l_B(\gamma) \ge d_{\Teich}\Big( \widetilde{F}(\widetilde{t}), F_{*}([\gamma])\cdot \widetilde{F}(\widetilde{t})  \Big)$.
\end{proposition}

\begin{proposition}
\label{proposition::distance-decreasing-2}
Let $\gamma\subset B$ be a loop based at $t$
and $\gamma'\subset B$ be a free loop homotopic to $\gamma$.
Then $(1/2) l_B(\gamma') \ge \tau( F_{*}([\gamma]) )$.
\end{proposition}

As a consequence, we have the following.

\begin{corollary}
\label{corollary::peripheral monodromy}
Each peripheral monodromy $\phi$ satisfies $\tau(\phi)=0$.    
\end{corollary}

The necessity for peripheral monodromies having infinite order in Theorem \ref{thmx::quasi-isometry} is the following.

\begin{proposition}
\label{proposition::torsion-peripheral}
Let $F:B\rightarrow \mathcal{M}_h$ be a holomorphic map.
If some peripheral monodromy of $F$ has finite order, then the restriction of $F$ to the corresponding cusp region is neither a quasi-isometric embedding nor a quasi-isometric immersion.
\end{proposition}

\begin{proof}
Suppose that the peripheral monodromy around some cusp point $\star$ of $B$ has finite order.
Let $U \subset B$ be the cusp neighbourhood of $\star$ bihomorphic to the punctured disc $\mathbb{D}^*$.
By \cite[Theorem 2]{Imayoshi1981}
(cf. \cite[Theorem A (2)]{Imayoshi-2009}),
the map $\left.F\right|_U$ extends to a holomorphic map $F_{\overline{U}}:\overline{U} \rightarrow \mathcal{M}_h$, where $\overline{U} \coloneqq U \sqcup \{\star\}$ is biholomorphic to the disc $\mathbb{D}$.

The pullback norm $\left( F_{\overline{U}} \right)^* \Kob_{\mathcal{M}_h}$ induces a distance $d_{F_{\overline{U}}}$ on $\overline{U}$ with
$d_F(b_1,b_2) = d_{F_{\overline{U}}}(b_1,b_2)$ for all $b_1, b_2 \in U$.
The compactness of $\overline{U}$ implies that
$d_F(b_1,b_2) \le \sup_{b_1,b_2 \in \overline{U}} d_{F_{\overline{U}}}(b_1,b_2) \eqqcolon \diam_{F_{\overline{U}}}(\overline{U})$.
Consider points $b_n \in U$ lying on the horocycle of length $1/n$.
We get that 
$d_B(b_1,b_n) \rightarrow \infty$
while $d_F(b_1,b_n) \le \diam_{F_{\overline{U}}}(\overline{U})$.
This shows that the restriction $\left.F\right|_U$ 
fails to be a quasi-isometric immersion and, consequently, cannot be a quasi-isometric embedding.
\end{proof}

\subsection{Essentially purely pseudo-Anosov monodromy}
\label{subsection::eppA}

We introduce the following hypothesis on the monodromy homomorphism (see also \cite{Reid2006}).

\begin{definition}
\label{definition::eppA}
Let $B$ be an oriented hyperbolic surface of type $(g,n)$ and $F:B\rightarrow \mathcal{M}_h$ be a holomorphic map.
We say that a monodromy homomorphism $F_*:\pi_1(B,t)\rightarrow \Mod_h$ of $F$ is \emph{essentially purely pseudo-Anosov} if for each non-trivial non-peripheral class $[\gamma]\in \pi_1(B,t)$ the image $F_*([\gamma])$ is pseudo-Anosov.
\end{definition}

Isometrically immersed holomorphic curves are specific examples whose monodromy homomorphisms are essentially purely pseudo-Anosov.
In fact, as mentioned in \cite{EskinMirz2011}, every closed geodesic in  $\mathcal{M}_h$ is the unique loop of minimal length in its homotopy class.
Hence, we have the following.

\begin{theorem}
\label{theorem::Teichmuller curve is eppA}
Each monodromy homomorphism of a Teichm\"uller curve is essentially purely pseudo-Anosov.
\end{theorem}

For the monodromy of a holomorphic map,
the property of being essentially purely pseudo-Anosov is stronger than the hypotheses required in Theorem \ref{thmx::quasi-isometry} and Theorem \ref{thmx::quasi-isometry-fundamental-domain-n=0}.
Before providing a detailed explanation, we first introduce the following definitions for mapping classes and peripheral monodromies.
These concepts will be used in Theorem \ref{theorem::quasi-isometry-fundamental-domain}, which serves as the general version of Theorem \ref{thmx::quasi-isometry-fundamental-domain-n=0}.

\begin{figure}[htb]
\centering
\begin{minipage}{.48\textwidth}
  \centering
  \def\svgwidth{.9\textwidth}
\begingroup%
  \makeatletter%
  \providecommand\color[2][]{%
    \errmessage{(Inkscape) Color is used for the text in Inkscape, but the package 'color.sty' is not loaded}%
    \renewcommand\color[2][]{}%
  }%
  \providecommand\transparent[1]{%
    \errmessage{(Inkscape) Transparency is used (non-zero) for the text in Inkscape, but the package 'transparent.sty' is not loaded}%
    \renewcommand\transparent[1]{}%
  }%
  \providecommand\rotatebox[2]{#2}%
  \newcommand*\fsize{\dimexpr\f@size pt\relax}%
  \newcommand*\lineheight[1]{\fontsize{\fsize}{#1\fsize}\selectfont}%
  \ifx\svgwidth\undefined%
    \setlength{\unitlength}{566.92913386bp}%
    \ifx\svgscale\undefined%
      \relax%
    \else%
      \setlength{\unitlength}{\unitlength * \real{\svgscale}}%
    \fi%
  \else%
    \setlength{\unitlength}{\svgwidth}%
  \fi%
  \global\let\svgwidth\undefined%
  \global\let\svgscale\undefined%
  \makeatother%
  \begin{picture}(1,1.2)%
    \lineheight{1}%
    \setlength\tabcolsep{0pt}%
    \put(0,0){\includegraphics[width=\unitlength,page=1]{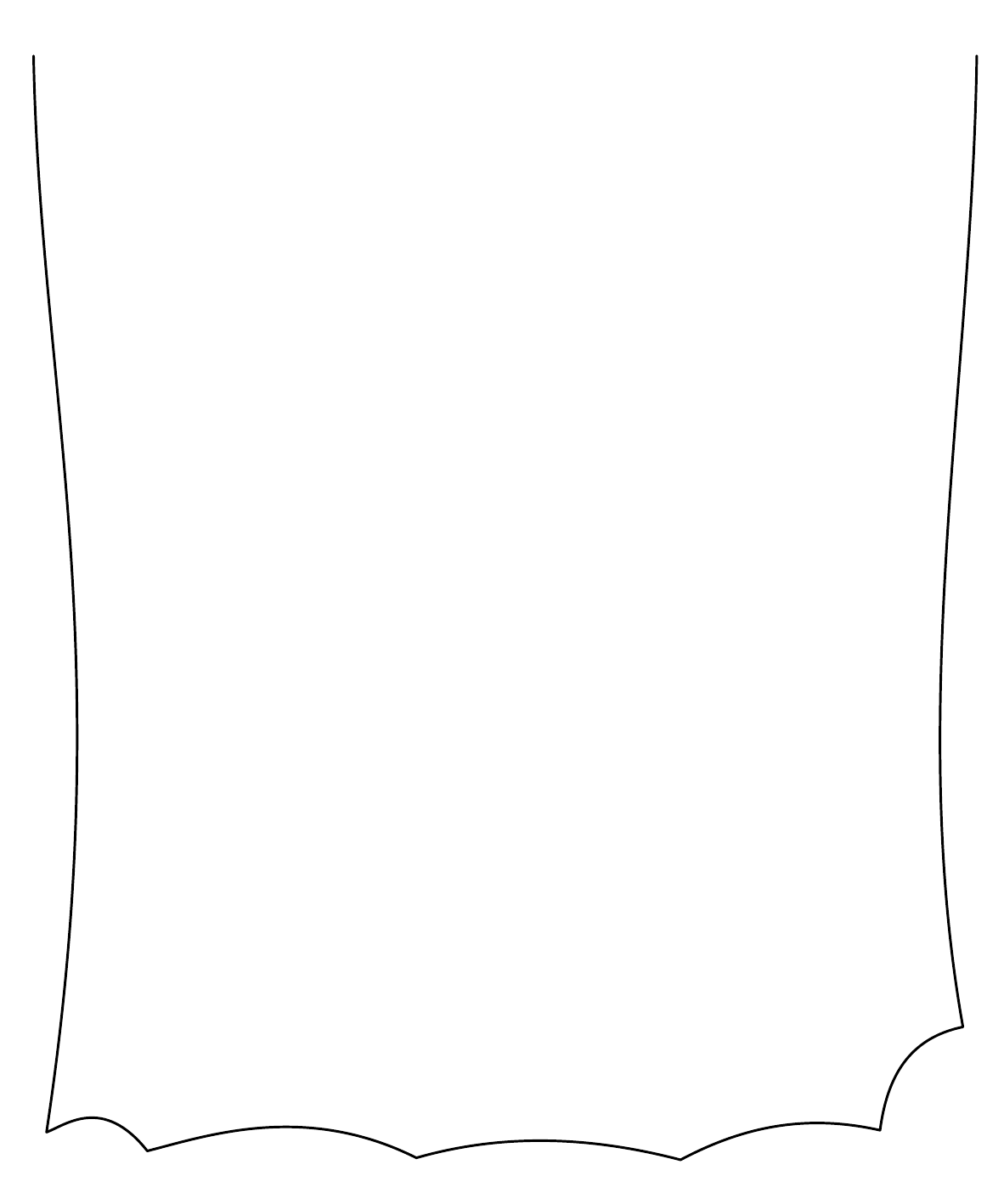}}%
    \put(0.07116224,0.10313347){\color[rgb]{0,0,0}\transparent{0.99999964}\makebox(0,0)[lt]{\lineheight{1.25}\smash{\begin{tabular}[t]{l}$\mathcal{M}_h$\end{tabular}}}}%
    \put(0,0){\includegraphics[width=\unitlength,page=2]{figures/holomorphic-curve.pdf}}%
  \end{picture}%
\endgroup%

  \caption{A cartoon of a holomorphic curve in $\mathcal{M}_h$. Each cusp neighbourhood of the holomorphic curve almost keeps the hyperbolic structure along the Teichm\"uller metric, unless the peripheral monodromy is of finite order.}
  \label{figure::holomorphic-curve}
\end{minipage}\hfill
\begin{minipage}{.48\textwidth}
  \centering
  \def\svgwidth{.9\textwidth}
  \input{pdf_tex/kappa.pdf_tex}
  \caption{Different geodesic segments joining the boundaries of two given cusp regions $U_i$ and $U_j$, each providing a pair of peripheral monodromies $(\phi_i,\phi_j)$.}
  \label{figure::kappa}
\end{minipage}
\end{figure}

\begin{definition}[\emph{disjointed} mapping classes]
Let $\phi_1$ and $\phi_2\in \Mod_h$ be reducible mapping classes.
Suppose that there exist positive integers $\mu_1$, $\mu_2$ and multi-curves
$\boldsymbol{\alpha_1}=\{{\alpha_{1,1},\ldots,\alpha_{1,m_1}}\}$, $\boldsymbol{\alpha_2}=\{{\alpha_{2,1},\ldots,\alpha_{2,m_2}}\}$ 
such that $\phi_1^{\mu_1}$, $\phi_2^{\mu_2}$ are multi-twists along $\boldsymbol{\alpha_1}$, $\boldsymbol{\alpha_2}$,
\begin{equation*}
\phi_1^{\mu_1}=
T_{\alpha_{1,1}}^{r_{1,1}}\circ
\cdots \circ 
T_{\alpha_{1,m_1}}^{r_{1,m_1}},
~~~~~
\phi_2^{\mu_2}=
T_{\alpha_{2,1}}^{r_{2,1}}\circ
\cdots \circ 
T_{\alpha_{2,m_1}}^{r_{2,m_1}}
\end{equation*}
with $r_{i,j}\in\Z\setminus\{0\}$, for $i=1,2$ and $j=1,\ldots,m_i$.
We say that $\phi_1$ and $\phi_2$ are \emph{disjointed} if
pairs of curves in $\boldsymbol{\alpha_1}$ and $\boldsymbol{\alpha_2}$ are disjoint or coincide.
\end{definition}

Consider a holomorphic map $F:B\rightarrow \mathcal{M}_h$ and the monodromy $F_{*}: \pi_1(B,t)\rightarrow \Mod_h$.

\begin{definition}[\emph{disjointed} peripheral monodromies]
Let $U_i$, $U_j$ be cusp regions of the hyperbolic surface $B$, $i\neq j$, endowed with a geodesic segment $\kappa$ joining $\partial U_i$ to $\partial U_j$.
Set $\{t_0\}=\partial U_i \cap \kappa$.
and take an arbitrary path $\gamma$ joining $t$ to $t_0$ (see Figure \ref{figure::kappa}).
The loop along $\gamma\cup \partial U_i$ based at $t$ that goes once around $U_i$ clockwise is denoted by $\gamma_i$ and its monodromy is denoted by $\phi_i$.
The loop along $\gamma\cup \kappa\cup \partial U_j$ based at $t$ that goes once around $U_j$ clockwise is denoted by $\gamma_j$ and its monodromy is denoted by $\phi_j$.
We say that peripheral monodromies of $U_i$ and $U_j$ are \emph{disjointed} along $\kappa$
if $\phi_i$ and $\phi_j$ are reducible and disjointed.
\end{definition}

A pair of disjointed mapping classes after a simultaneous conjugacy is again disjointed.
Therefore, a pair of peripheral monodromies being disjointed along $\kappa$ is independent of the choice of the path $\gamma$.
However, a different choice of $\kappa$ changes the peripheral monodromies by a non-simultaneous conjugacy.

\begin{proposition}
\label{proposition::faithful monod repr+B bounded systole}
Let $B=\Gamma \backslash \Hyperbolic^2$ be an oriented hyperbolic surface of type $(g,n)$.
Let $F:B\rightarrow \mathcal{M}_h$ be a holomorphic map with an essentially purely pseudo-Anosov monodromy homomorphism $F_* \in \Hom(\pi_1(B,t),Mod_h)$. Then
\begin{enumerate}[(i)]
    \item $F$ is non-constant;
    
    \item the monodromy homomorphism is injective;

    \item each peripheral monodromy is of infinite order;

    \item peripheral monodromies of any two cusps are not disjointed along any geodesic segment $\kappa$ between the boundaries of their cusp regions;
    
    \item $\sys(B)\ge 2\log{2}/(12h-12)$.
\end{enumerate}
\end{proposition}

\begin{proof}
For (\RNum{1}), we notice that there exists at least one non-trivial non-peripheral element in $\pi_1(B,t)$.
For (\RNum{2}), it suffices to show that each peripheral element cannot be represented by the identity.
Take the group presentation
\begin{equation*}
    \pi_1(B,t)=\langle a_1,b_1,\ldots,a_g,b_g,c_1,\ldots,c_n \mid \prod_i [a_i,b_i] \prod_j c_j = 1 \rangle.
\end{equation*}
with loops $a_1,\ldots,a_g$, $b_1,\ldots,b_g$,$c_1,\ldots,c_n$ at $t$.
Consider a positive power of a peripheral generating loop, say $c_j^r$ with $j=1,\ldots,n$.
When $g\ge 1$ and $n\ge 1$, then $[c_j^r,a_1]\neq 1$ is non-peripheral.
When $n\ge 2$, take $j'\neq j$ and then $[c_j^r,c_{j'}]\neq 1$ is non-peripheral.
In both cases, we get $F_*(c_j^r)\neq 1$.
For (\RNum{3}), a peripheral monodromy $\phi$ must be of infinite order due to the injectivity of the monodromy homomorphism. 

For (\RNum{4}),
let $U_1$, $U_2$ be two distinct cusp regions of $B$ linked by a geodesic segment $\kappa$.
Set ${t_0}=\partial U_1\cap \kappa$ and take an arbitrary path $\gamma$ joining $t$ to $t_0$.
The loop along $\gamma \cup \partial U_1$ based at $t$ that goes once around $U_1$ clockwise is denoted by $\gamma_1$ and its monodromy is denoted by $\phi_1$.
The loop along $\gamma \cup \kappa \cup \partial U_2$ based at $t$ that goes once around $U_2$ clockwise is denoted by $\gamma_2$ and its monodromy is denoted by $\phi_2$.
By Corollary \ref{corollary::peripheral monodromy}, some power $\phi_1^{\mu_1}$ is the multi-twist along a multi-curve $\boldsymbol{\alpha_1}$ and some power $\phi_2^{\mu_2}$ is the multi-twist along a multi-curve $\boldsymbol{\alpha_2}$.
Assume that $\boldsymbol{\alpha}_1\cup\boldsymbol{\alpha}_2$ is a set of disjoint simple closed curves.
We notice that $\gamma_1^{2\mu_1}\cdot \gamma_2^{2\mu_2}$ is non-peripheral of which the monodromy is reducible, which is a contradiction.
The parameters $2$ in $\gamma_1^{2\mu_1}\cdot \gamma_2^{2\mu_2}$ cannot be replaced with $1$.
Indeed, when $(g,n)=(0,3)$ and $\mu_1=\mu_2=1$, $\gamma_1\cdot \gamma_2$ is peripheral but
$\gamma_1^2\cdot \gamma_2^2$ not.

For (\RNum{5}), consider an essential closed geodesic $\gamma\subset B$.
Then there exists a loop $\gamma'$ based at $t$ homotopic to $\gamma$, which is non-trivial and non-peripheral.
By Proposition \ref{proposition::distance-decreasing-2}, we have $(1/2) l_B(\gamma)\ge \tau(F_*([\gamma']))$.
Penner proved in \cite[p.444]{Penner1991} the inequality $\tau(\phi) \ge \log{2}/(12h-12)$, for any pseudo-Anosov mapping class $\phi\in \Mod_h$ (see also \cite[Theorem 14.10]{FarbMargalit2011}).
Thus, $\sys(B)\ge 2\log{2}/(12h-12)$.
\end{proof}

\begin{remark}
One can apply Theorem \ref{thmx::quasi-isometry} - (\RNum{2}) and Theorem \ref{thmx::quasi-isometry-fundamental-domain-n=0} (and Theorem \ref{theorem::quasi-isometry-fundamental-domain} for the general case) to a holomorphic map $F:B\rightarrow \mathcal{M}_h$ with an essentially purely pseudo-Anosov monodromy homomorphism.
Then $F$ is a quasi-isometric immersion with parameters depending only on $(g,n)$ and $h$.
In addition, the lift $\widetilde{F}:\Hyperbolic^2\rightarrow \Teich_h$ restricted to any convex fundamental polygon $D$ with exactly $n$ ideal points is a quasi-isometric embedding.
\end{remark}

\subsection{Mumford’s compactness}
\label{subsection::Mumford}

Let $\epsilon>0$ be an arbitrary real number.
For the Teichm\"uller metric, the moduli space $\mathcal{M}_{g,n}$ has an infinite diameter. In \cite{Mumford1971}, however, Mumford introduced that the $\epsilon$-thick part of the moduli space that consists of hyperbolic surfaces $X$ with $\sys(X)\ge \epsilon$ is compact.

Let $\Teich_{g,n}^{\ge \epsilon}$ be the set of equivalent classes of
marked hyperbolic surfaces whose systole is bounded below by $\epsilon$.
The action of $\Mod_{g,n}$ on $\Teich_{g,n}$ preserves the systole and therefore we take the quotient space of the thick part, denoted by $\mathcal{M}_{g,n}^{\ge \epsilon}$.
We have the following compactness result due to Mumford:

\begin{theorem}[Mumford]
The $\epsilon$-thick part $\mathcal{M}_{g,n}^{\ge \epsilon}$ of the moduli space $\mathcal{M}_{g,n}$ is a compact subset.
\end{theorem}

Form now on, we fix a base point $[(X_0,f_{X_0})]$ in $\Teich_{g,n}$ that is represented by a marked complete hyperbolic surface $(X_0,f_{X_0})$.
We fix a base point $s\in\Sigma_{g,n}$ and fix the oriented loops
\begin{equation*}
\gamma_{1,1},\gamma_{1,2},\ldots,
\gamma_{g,1},\gamma_{g,2},
\gamma'_1,\ldots,\gamma'_n
\subset \Sigma_{g.n}
\end{equation*}
at $s$ satisfying the following conditions:
(\RNum{1}) $\Int(\gamma_{i,1},\gamma_{i,2})=1$, for each $i$, but the intersection number of any other two distinct loops is $0$,
(\RNum{2}) $\gamma'_j$ goes round the $j$-th puncture exactly once clockwise, for each $j$, and
(\RNum{3}) the fundamental group $\pi_1(\Sigma_{g,n},s)$ is generated by classes of these loops with the relation
$\prod_{i=1}^g [[\gamma_{i,1}],[\gamma_{i,2}]] \prod_{j=1}^n [\gamma'_j]=1$.
The loops $\gamma_{i,j}$, $\gamma'_k$ are called the \emph{standard loops} of $\Sigma_{g,n}$.

In contrast, a collection of loops 
$\widehat{\gamma_{i,1}}$, $\widehat{\gamma_{i,2}}$ $\widehat{\gamma'_j} \subset \Sigma_{g,n}$,
for $i=1,\ldots, g$, $j=1,\ldots,n$, based at some point $\widehat{s} \in\Sigma_{g,n}$ satisfying conditions (\RNum{1}), (\RNum{2}) and (\RNum{3}) implies an orientation preserving diffeomorphism $\phi:\Sigma_{g,n}\rightarrow\Sigma_{g,n}$ such that $\phi(s)=\widehat{s}$ and each $\phi(\gamma_{i,1})$ (resp. $\phi(\gamma_{i,2})$, $\phi(\gamma_j)$) is homotopic to $\widehat{\gamma_{i,1}}$ (resp. $\widehat{\gamma_{i,2}}$, $\widehat{\gamma'_j}$) relative to $\widehat{s}$.

Set $t_0=f_{X_0}(s)\in X_0$.
The set $\{[f_{X_0}(\gamma)]\mid \text{$\gamma$ is a standard loop}\}$ forms a generating set of $\pi_1(X_0,t_0)$ such that the length of each $f_{X_0}(\gamma)$ is determined by $(g,n)$.
In other words, there exists a constant $N(g,n)$ that depends only on $(g,n)$ such that
\begin{equation*}
    l_{X_0}(f_{X_0}(\gamma)) \le N(g,n)
\end{equation*}
for each standard loop $\gamma$ of $\Sigma_{g,n}$.
We further notice the following lemma.

\begin{lemma}
    Suppose that $(Y,f_Y)$ is a marked hyperbolic surface and $\psi:X_0\rightarrow Y$ is a $K$-quasiconformal diffeomorphism. Then there exists an orientation preserving diffeomorphism $f'_Y:\Sigma_{g,n}\rightarrow Y$ such that,
    for each standard loop $\gamma$,
    the image $f'_Y(\gamma)\subset Y$ is homotopic to a loop of length bounded above by $N'(g,n,K)$ relative to $f'_Y(s)$,
    where $N'(g,n,K)$ depends only on $g$, $n$ and $K$.
\end{lemma}

\begin{proof}
One can construct $2g+n+[n=0]$ free loops on $\Sigma_{g,n}$, where $[n=0]=1$ if $n=0$ and $[n=0]=0$ otherwise.
These loops, denoted by $\delta_0,\ldots,\delta_{2g+n-1+[n=0]}$, form a set of closed curves that fill up the surface $\Sigma_{g,n}$ and
each $\delta_i$ is the concatenation of at most four standard loops.

Let $\Delta_Y$ be the union of geodesics homotopic to every $\psi\circ f_{X_0}(\delta_i)$. By Wolpert's Lemma \ref{lemma::wolpert},
\begin{equation*}
    l_Y(\Delta_Y)=\sum_{i=0}^{2g+n-1+[n=0]} L_{\delta_i}([(Y,f_Y)])\le e^{2K}\sum_{i=0}^{2g+n-1+[n=0]} L_{\delta_i}([(X_0,f_{X_0})])
    \le 4(2g+n)e^{2K}\cdot N(g,n).
\end{equation*}
Let $t_Y\in\Delta_Y$ be an arbitrary point.
Therefore, there exists loops $\widehat{\gamma_{i,1}},\widehat{\gamma_{i,2}}$ and $\widehat{{\gamma'}_j}$ at $t_Y$ on $Y$, for $i=1,\ldots,g$ and $j=1,\ldots,n$, that satisfy conditions (\RNum{1}), (\RNum{2}) and (\RNum{3}), whose lengths are bounded above by
$4\cdot 4(2g+n)e^{2K}N(g,n)
\eqqcolon N'$.
We conclude that there exists a diffeomorphism $f'_Y$ sending $s\in\Sigma_{g,n}$ to $t_Y$ such that, for each standard loop $\gamma$, the homotopy class of $f'_Y(\gamma)$ relative to $t_Y$ is represented by a loop of length bounded above by $N'$.
\end{proof}

We aim at addressing the following question:
given a hyperbolic surface $X$ of type $(g,n)$, does there exist an orientation preserving diffeomorphism $f'_X:\Sigma_{g,n}\rightarrow X$ whose image of each standard loop is homotopic to a loop of short length relative to $f'_X(s)$?




Consider a non-trivial mapping class $[\phi]$ represented by a diffeomorphism $\phi:\Sigma_{g,n}\rightarrow\Sigma_{g,n}$.
Then, the marked hyperbolic surfaces $(X,f_X)$ and $(X,f_X\circ\phi^{-1})$ are the same hyperbolic surface with different markings.
Therefore, they possess desired diffeomorphisms for this question simultaneously.
Using Mumford's compactness of $\mathcal{M}_{g,n}^{\ge\epsilon}$, we get the following theorem.

\begin{theorem}
\label{theorem::question 2_for_b=0}
There exists a constant $N''=N''(g,n,\epsilon)$ that depends only on $(g,n)$, $\epsilon$ and satisfies the following:
Given an arbitrary hyperbolic surface $Y$ of type $(g,n)$ with $\sys(Y)\ge \epsilon$,
there exists an orientation preserving diffeomorphism $f'_Y:\Sigma_{g,n}\rightarrow Y$ such that for each standard loop $\gamma\subset\Sigma_{g,n}$,
the image $f'_Y(\gamma)$ is homotopic to a loop whose length is bounded above by $N''$ relative to $f'_Y(s)$.
\end{theorem}

\begin{proof}
By Mumford's compactness, the moduli space $\mathcal{M}_{g,n}^{\ge\epsilon}$ is compact and
therefore there exists a lift $U\subset \Teich_{g,n}$ of $\mathcal{M}_{g,n}^{\ge\epsilon}$ containing $[X_0,f_{X_0}]$ whose diameter for the Teichm\"uller metric is bounded by a constant $D=D(g,n,\epsilon)>0$ that depends only on $(g,n)$ and $\epsilon$.
Suppose that $f_Y$ is an arbitrary marking of $Y$ so that $[(Y,f_Y)]\in\Teich_{g,n}^{\ge\epsilon}$.
Then, there exists an orientation preserving diffeomorphism $\phi:\Sigma_{g,n}\rightarrow\Sigma_{g,n}$ with $[(Y,f_Y\circ\phi^{-1})]\in U$.
Moreover, there exists a $K$-quasiconformal diffeomorphism $\psi:X_0\rightarrow Y$ homotopic to $(f_Y\circ\phi^{-1}) \circ f_{X_0}^{-1}$, where $e^{2K}\le D$.
Let $f'_Y:\Sigma_{g,n}\rightarrow Y$ be the diffeomorphism as in the previous lemma and
take $N''=N'(g,n,K)$.
Therefore, for each standard loop $\gamma$ on $\Sigma_{g,n}$, the image $f'_Y(\gamma)$ is homotopic to a loop of length bounded above by $N''$ relative to $f'_Y(s)$.
\end{proof}

\section{Quasi-isometric rigidity}

Suppose that
$B$ is a hyperbolic surface of type $(g,n)$ and
$F:B\rightarrow \mathcal{M}_h$ is a holomorphic map.
In this section, we introduce the rigidity result, which claims that the holomorphic curve $F(B)$ is very similar to a Teichm\"uller curve.

\subsection{Irreducibility of holomorphic curves}
\label{subsection::irreducibility}

We say a representation $\rho: \pi_1(B,t) \rightarrow \Mod_h$ is \emph{reducible} if the image subgroup $\rho(\pi_1(B,t))$ preserves a set of homotopy classes
$\{ [\alpha_1],\ldots,[\alpha_{h'}] \}$
represented by disjoint simple closed curves on $\Sigma_h$, with $h' \ge 1$.

The following proposition is known as the \emph{irreducibility} of holomorphic curves in the moduli space.
It first appeared in \cite{ImayoshiShiga1988} as part of the proof of the Parshin-Arakelov finiteness.
Shiga later formally formulated this result in \cite{Shiga1997}.
For the case of $n=0$, the irreducibility is a consequence of \cite[Theorem 5.7]{Daskalopoulos-Wentworth2007} and
is also proved using the maximum principle for subharmonic functions \cite[Theorem 3.1]{McMullen2000}).

\begin{proposition}[Theorem 1 in \cite{Shiga1997}]
\label{proposition::no_disjoint_simple_loops_preserved}
Let $B$ be an oriented hyperbolic surface of type $(g,n)$ and 
$F:B\rightarrow \mathcal{M}_h$ be a non-constant holomorphic map.
Then, every monodromy homomorphism
$F_*:\pi_1(B,t)\rightarrow \Mod_h$
of $F$ is irreducible.
\end{proposition}

By this irreducibility, we have the following.

\begin{theorem}
\label{theorem::Ct_belongs_to_compactK}
There exist a constant $N''=N''(g,n,\epsilon)$ depending only on $(g,n)$, $\epsilon$ and a compact subset $\mathcal{K}'=\mathcal{K}'(g,n,h,\epsilon)\subset\mathcal{M}_h$ depending only on $(g,n),h,\epsilon$, that satisfy the following:
Let $B$ be an oriented hyperbolic surface of type $(g,n)$ with $\sys(B)\ge \epsilon$
and let $F:B\rightarrow \mathcal{M}_h$ be a non-constant holomorphic map.
Then, there exists an orientation preserving diffeomorphism $f'_B:\Sigma_{g,n}\rightarrow B$ such that
\begin{enumerate}[(i)]
\item $f'_B$ maps the base point $s\in\Sigma_{g,n}$ to $t'_B\in B$, with $F(t'_B)\in \mathcal{K}'$;
\item for each standard loop $\gamma\subset\Sigma_{g,n}$, the image $f'_B(\gamma)\subset B$ is homotopic to a loop relative to $t'_B$ with length bounded above by $N''$.
\end{enumerate}
\end{theorem}

Theorem \ref{theorem::Ct_belongs_to_compactK} is a continuation of Theorem \ref{theorem::question 2_for_b=0},
for which
the constant $N''(g,n,\epsilon)$ and the orientation preserving diffeomorphism $f'_B$ are inherited.
We start with a well-known lemma on closed hyperbolic surface, which will imply that $\sys(F(t'_B))$ is bounded.
The proof of Theorem \ref{theorem::Ct_belongs_to_compactK} proceeds similarly to the proof of \cite[Theorem 3.1]{McMullen2000}.

\begin{lemma}[Corollary 13.7 in \cite{FarbMargalit2011}]
\label{lemma::3h-3_simple_loops}
    There exists a constant $\tau_h>0$ depending only on $h$ such that on any closed oriented hyperbolic surface $X$ of genus $h$, if $\{\alpha_1,\ldots,\alpha_{h'}\}$ is the set of closed geodesics of length smaller than $\tau_h$, then $\Int(\alpha_i,\alpha_j)=0$ and $h'\le 3h-3$.
\end{lemma}

\begin{proof}[Proof of Theorem \ref{theorem::Ct_belongs_to_compactK}]
We aim to show that $\sys(F(t'_B))\ge \tau_h/N''^{3h-3}$ and therefore take $\mathcal{K}'=\mathcal{M}_h^{\ge \tau_h/N''^{3h-3}}$, which is a compact subset by Mumford’s compactness.
Suppose that $F(t'_B)$ has a hyperbolic representative $(X,f_X)$ and assume that $\sys(X)<\tau_h/N''^{3h-3}$.
By Lemma \ref{lemma::3h-3_simple_loops},
let $\alpha_1,\ldots,\alpha_{h'}$ be closed geodesics on $X$ of length smaller than $\tau_h$, with $h'\le 3h-3$.

Without loss of generality, assume $l_X(\alpha_1) \le \cdots \le l_X(\alpha_{h'})$.
Since $h' \le 3h-3$ while $\sys(X) = l_X(\alpha_1)$, there must exist a significant length gap, i.e., $N'' \cdot l_X(\alpha_i) < l_X(\alpha_{i+1})$ for some index $i$.
By Wolpert's Lemma \ref{lemma::wolpert},
the preimages $f_X^{-1}(\alpha_1), \ldots, f_X^{-1}(\alpha_i)$
form a set of homotopy classes on $\Sigma_h$
that is preserved by $F_*(g)\in\Mod_h$ for each $g\in \pi_1(B,t'_B)$,
which contradicts Proposition \ref{proposition::no_disjoint_simple_loops_preserved}.
\end{proof}

\begin{remark}
Consequently, we have the uniform boundedness for Parshin-Arakelov finiteness as follows.
Fix $s\in \Sigma_{g,n}$.
Given a holomorphic map $F:B\rightarrow \mathcal{M}_h$ and a monodromy homomorphism $F_*\in \Hom(\pi_1(B,t),\Mod_h)$,
there exists an orientation preserving diffeomorphism $f_B:\Sigma_{g,n}\rightarrow B$ such that $f_B(s)=t$.
Therefore, $f_B$ pulls $F_*$ back to $(f_B)^*(F_*)=F_*\circ (f_B)_* \in \Hom(\pi_1(\Sigma_{g,n},s),\Mod_h)$.
A different choice of $f_B$, however, does not change the corresponding class in
\begin{equation}
\label{equation::double_quotient_space}
    M_{g,n,h}\coloneqq \Mod_{g,n} \backslash \Hom(\pi_1(\Sigma_{g,n},s), \Mod_h) / \Mod_h.
\end{equation}
Thus to a holomorphic map $F:B\rightarrow\mathcal{M}_h$ one associates the class
$\MO(F)\coloneqq [F_*\circ (f_B)_*] \in M_{g,n,h}$.
Theorem \ref{theorem::Ct_belongs_to_compactK} implies that, given $\epsilon > 0$, the subset
\begin{equation*}
\left\{
\MO(F)
\middle|
\begin{aligned}
&B
\text{ is an oriented hyperbolic surface of type $(g,n)$ such that $\sys(B)\ge \epsilon$}
\\
&
F:B\rightarrow \mathcal{M}_h
\text{ is a non-constant holomorphic map}
\end{aligned}
\right\}
\subset M_{g,n,h}
\end{equation*}
is finite, where the finiteness
depends only on $g$, $n$, $h$ and $\epsilon$ (cf. \cite{Caporaso}).
\end{remark}

\subsection{From cusp region to end of moduli space}
\label{subsection::cusp}

The moduli space $\mathcal{M}_h$ has only one end, meaning that for any compact set, there is
exactly one unbounded component of the complement.
In this subsection, we consider a hyperbolic cusp region $U$, i.e. the neighbourhood of a cusp bounded by a horocycle of length $2$.
Then we investigate a non-constant map $F:U\rightarrow \mathcal{M}_h$ that is distance-decreasing for $1/2$ of the hyperbolic distance $d_{U}$ on $U$ and the Teichm\"uller distance $d_{\mathcal{M}}$ on $\mathcal{M}_h$.

The Dehn twist along a closed curve $\alpha\subset\Sigma_h$, denoted by $\tau_{\alpha}$, is a diffeomorphism of $\Sigma_h$ and represents a mapping class $T_{\alpha}\in\Mod_h$.
Let $\boldsymbol{\alpha}=\{\alpha_1,\ldots,\alpha_m\}$ be a multi-curve, then a multi-twist along $\boldsymbol{\alpha}$ is a product of the form
$T=T_{\alpha_1}^{r_1}\circ\cdots\circ T_{\alpha_m}^{r_m}\in\Mod_h$
with each $r_i\in \Z \setminus\{0\}$.
In particular, a power of positive or negative Dehn twist is a multi-twist in our discussion.

\begin{lemma}
\label{lemma::two points partitioned by two geodesics}
Let $\widetilde{\gamma_1},\widetilde{\gamma_2}\subset\Hyperbolic^2$ be disjoint geodesics
and $k\ge 3$ be an integer.
Suppose that $\phi$ is a hyperbolic isometry along $\widetilde{\gamma_1}$ whose translation length is equal to $l$
such that $\widetilde{\gamma_2}\cap \phi(\widetilde{\gamma_2})=\emptyset$.
Let $\widetilde{p_1},\widetilde{p_2}\in\Hyperbolic^2$ be arbitrary points separated by both $\widetilde{\gamma_2}$ and $\phi^k(\widetilde{\gamma_2})$.
Then $d_{\Hyperbolic^2}(\widetilde{p_1},\widetilde{p_2})\ge l$.
\end{lemma}

\begin{proof}
Take the (unique) geodesic segment $\beta_i$ perpendicular to both $\widetilde{\gamma_1}$ and $\phi^i(\widetilde{\gamma_2})$, for $i=1,2$.
The hyperbolic plane is separated by 
$\widetilde{\gamma_1}$, $\widetilde{\gamma_2}$, $\phi(\widetilde{\gamma_2})$, $\phi^2(\widetilde{\gamma_2})$, $\phi^k(\widetilde{\gamma_2})$
and $\beta_1$, $\beta_2$ into $8$ pieces.
Therefore, the geodesic segment joining $\widetilde{p_1}$ to $\widetilde{p_2}$ goes cross $\beta_1$ and $\beta_2$.
Hence the distance between $\widetilde{p_1}$ and $\widetilde{p_2}$ is at least the distance between $\beta_1$ and $\beta_2$, which is equal to $l$.
\end{proof}

\begin{lemma}[Lemma 4.2 in \cite{Ivanov1992}, cf. Expos\'e 6, Section \URNum{7} in \cite{FLP}]
\label{lemma::ivanov}
Let
$\tau=\tau_{\alpha_1}^{r_1}\circ\cdots\circ \tau_{\alpha_m}^{r_m}$ be a multi-twist diffeomorphism along $\boldsymbol{\alpha}=\{\alpha_1,\ldots,\alpha_m\}$.
Then for all closed curves $\gamma_1,\gamma_2\subset \Sigma_h$, we have
\begin{equation*}
\Int(\tau(\gamma_1),\gamma_2)\ge 
\sum_{i=1}^m (|r_i|-2) \Int(\gamma_1,\alpha_i) \Int(\gamma_2,\alpha_i) - \Int(\gamma_1,\gamma_2).
\end{equation*}
\end{lemma}

In particular, for any multi-twist diffeomorphism $\tau$ along $\boldsymbol{\alpha}$ and closed curve $\gamma$ intersecting $\boldsymbol{\alpha}$ at least once, we have $\Int(\tau^3(\gamma),\gamma)\ge
\sum_{i=1}^m (|3r_i|-2) \Int(\gamma,\alpha_i)^2 \ge 1$.
The following lemma is inspired by this observation.

\begin{lemma}
\label{lemma::multi-twist_sum>=l_X(alpha_i)}
Let $T=T_{\alpha_1}^{r_1}\circ\cdots\circ T_{\alpha_m}^{r_m}\in\Mod_h$ be a multi-twist
along $\boldsymbol{\alpha}=\{\alpha_1,\ldots,\alpha_m\}$.
Let $\gamma\subset\Sigma_h$ be a simple closed curve such that
$\Int(\boldsymbol{\alpha},\gamma)\coloneqq \sum_i |r_i| \Int(\alpha_i,\gamma)\ge 1$.
Then, for any $[(X,f_X)]\in\Teich_h$, we have
\begin{equation*}
L_{\gamma}([(X,f_X)])+
L_{\gamma}(T^4\cdot [(X,f_X)])\ge
\frac{1}{3} L_{\alpha_i}([(X,f_X)])
\end{equation*}
for every $i\in \{1,\ldots,m\}$ satisfying $\Int(\alpha_i,\gamma)>0$.
\end{lemma}

\begin{proof}
Without loss of generality, we assume that $(X,f_X)$ is a hyperbolic representative of the given $[(X,f_X)]$.
Suppose that $T$ is represented by a multi-twist diffeomorphism $\tau$.
We have
\begin{equation*}
L_{\gamma}(T^4\cdot [(X,f_X)])=L_{\tau^{-4}(\gamma)}([(X,f_X)])
\end{equation*}
since $T\cdot [(X,f_X)]=[(X,f_X\circ \tau^{-1})]$.

Consider a universal covering $\iota:\Hyperbolic^2\rightarrow X$
such that the horizontal geodesic $\widetilde{\gamma_1}\subset\Hyperbolic^2$ is a lift of the (unique) geodesic homotopic to $f_X(\gamma)\subset X$ (see Figure \ref{figure::The lifts of gamma and T^4(gamma)}).
We suppose that $\widetilde{\gamma_1}$ is oriented toward the left.
Proceeding from $0\in\Hyperbolic^2$, the first lift of some $f_X(\alpha_u)$ that intersects $\widetilde{\gamma_1}$ is denoted by $\widetilde{\alpha_1}$ and
the second lift of some $f_X(\alpha_v)$ is denoted by $\widetilde{\alpha_2}$.
Going along the opposite direction, the first lift of some $f_X(\alpha_x)$ is denoted by $\widetilde{\alpha_{-1}}$ and the second lift of some $f_X(\alpha_y)$ is denoted by $\widetilde{\alpha_{-2}}$. 
Suppose that $\widetilde{\alpha_1}$ and $\widetilde{\alpha_{-1}}$ are oriented upward.

Fix $i=1,\ldots,m$.
Without loss of generality, we further assume that $u=i$.

The key point of the following argument is that while the exact position of $f_X(\tau^{-4}(\gamma)) \subset X$ is difficult to describe,
we can determine the endpoint on $\partial \Hyperbolic^2$ of any lift of $f_X(\tau^{-4}(\gamma))$ to $\Hyperbolic^2$.
More precisely, we start by fixing a complementary component of all lifts of $f_X(a_1),\ldots,f_X(a_m)$ in $\Hyperbolic^2$.
We then shear half of the plane along a boundary line of this component,
glue the two components along this boundary
and repeat this process iteratively.
Each lift of $\gamma$ yields a union of arcs that, when joined along the lifts of $f_X(\alpha_i)$, shares the same endpoints as the corresponding lift of $f_X(\tau^{-4}(\gamma))$.

Fix the complementary component bounded by $\widetilde{\alpha_1}$ and $\widetilde{\alpha_{-1}}$.
The shearing along $\widetilde{\alpha_1}$
that interprets $T_{\alpha_u}$ is a hyperbolic isometry $\phi_1\in\Isom(\Hyperbolic^2)$ acting on the left side of $\widetilde{\alpha_1}$,
whose translation length is $L_{\alpha_u}([(X,f_X)])$.
Therefore, the twist $T_{\alpha_u}^{r_u}$ is interpreted by $\phi_1^{r_u}$ acting on the left side of $\widetilde{\alpha_1}$.
Similarly,
the closed geodesic homotopic to $f_X(\alpha_x)\subset X$ is interpreted by a hyperbolic isometry $\phi_{-1}\in\Isom(\Hyperbolic^2)$
of which $\widetilde{\alpha_{-1}}$ is the axis, whose translation length is $L_{\alpha_x}([(X,f_X)])$.
Therefore, the twist $T_{\alpha_x}^{r_x}$ is interpreted by $\phi_{-1}^{-r_x}$ acting on the right side of $\widetilde{\alpha_{-1}}$.

\begin{figure}[H]
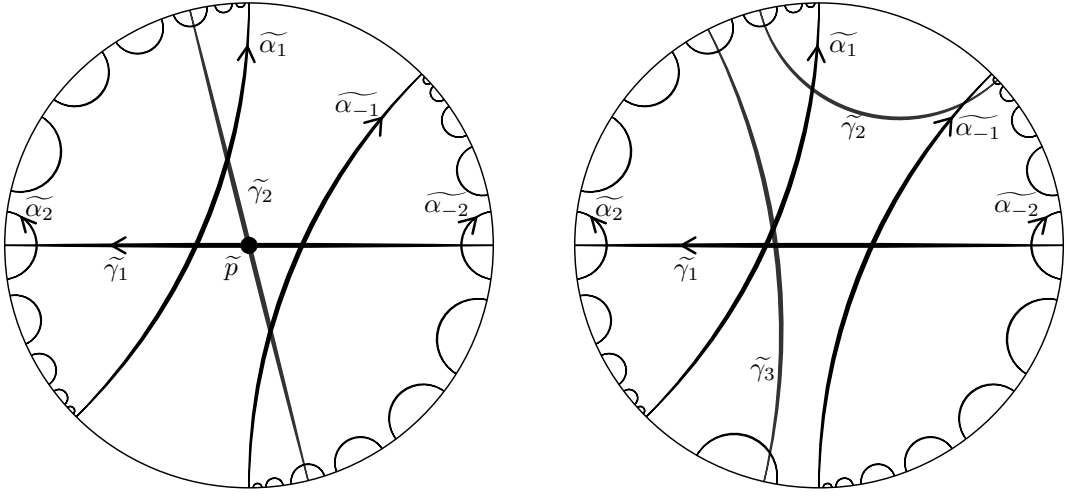

\centering
\begin{minipage}{.5\textwidth}
  \centering
  \def\svgwidth{.9\textwidth}
  \input{pdf_tex/multi-twist-lifting---left.pdf_tex}
\end{minipage}%
\begin{minipage}{.5\textwidth}
  \centering
  \def\svgwidth{.9\textwidth}
  \input{pdf_tex/multi-twist-lifting---right.pdf_tex}
\end{minipage}
\caption{Lifts of $\gamma$ and $T^4(\gamma)$ given a multi-twist $T$ along $\boldsymbol{\alpha}$.}
\label{figure::The lifts of gamma and T^4(gamma)}
\end{figure}

Suppose that the twists within $T$ along $\alpha_u$ and $\alpha_x$ have the same direction.
Without lose of generality, we assume that $r_u>0$ and $r_x>0$.
There exists a lift $\widetilde{\gamma_2}$ of the (unique) geodesic homotopic to $f_X(\tau^{-4}(\gamma))\subset X$ which connects two boundary points separated by both $\phi_1^{4r_u}(\widetilde{\alpha_2})$
and $\phi_{-1}^{-4r_x}(\widetilde{\alpha_{-2}})$.
Let $\widetilde{p}\in\Hyperbolic^2$ be the intersection $\widetilde{\gamma_1}\cap\widetilde{\gamma_2}$.
Let $\psi_1,\psi_2$ be hyperbolic isometries along $\widetilde{\gamma_1}$, $\widetilde{\gamma_2}$ corresponding to $f_X(\gamma)$, $f_X(\tau^{-4}(\gamma))$ respectively.
Therefore,
\begin{align*}
3\Big(
L_{\gamma}([(X,f_X)])+
L_{\gamma}(T^4\cdot [(X,f_X)])
\Big)
&=
3\Big(
L_{\gamma}([(X,f_X)])+
L_{\tau^{-4}(\gamma)}([(X,f_X)])
\Big)
\\
&=
d_{\Hyperbolic^2}(\widetilde{p},\psi_1^3(\widetilde{p}))
+
d_{\Hyperbolic^2}(\widetilde{p},\psi_2^3(\widetilde{p})).
\end{align*}
Meanwhile,
$\psi_1^3(\widetilde{p})$ and $\psi_2^3(\widetilde{p})$
are separated by both $\widetilde{\alpha_2}$ and
$\phi_1^{4r_u}(\widetilde{\alpha_2})$.
The desired inequality follows from Lemma \ref{lemma::two points partitioned by two geodesics}.

Suppose that twists within $T$ along $\alpha_u$ and $\alpha_x$ have different directions.
Without loss of generality, we assume that $r_u>0$ and $r_x<0$.
In this case, we have $\Int(\boldsymbol{\alpha},\gamma)\ge 2$.
Take the lift $\widetilde{\gamma_2}$ of the geodesic homotopic to $f_X(\tau^{-4}(\gamma))\subset X$ which connects two boundary points separated by both
$\phi_1^{4r_u}(\widetilde{\alpha_2})$
and
$\phi_{-1}^{-4r_x}(\widetilde{\alpha_{-2}})$.
Note that $\widetilde{\gamma_2}$ does not intersect $\widetilde{\gamma_1}$.
The geodesic $\widetilde{\gamma_3}=\phi_1^{-1}(\widetilde{\gamma_2})$ is another lift of the geodesic homotopic to $f_X(\tau^{-4}(\gamma))$,
which connects two boundary points that are separated by $\phi_1^{4r_u-1}(\widetilde{\alpha_2})$ and $\phi_1^{-1}(\widetilde{\alpha_{-1}})$.
Hence $\widetilde{\gamma_3}$ intersects $\widetilde{\gamma_1}$,
where the intersection is denoted by $\widetilde{p}$.
Let $\psi_1$, $\psi_3$ be hyperbolic isometries along $\widetilde{\gamma_1}$, $\widetilde{\gamma_3}$
corresponding to $f_X(\gamma)$, $f_X(\tau^{-4}(\gamma))$ respectively.
Therefore, $\psi_1^3(\widetilde{p})$ and $\psi_3^3(\widetilde{p})$ are separated by both $\widetilde{\alpha_2}$ and $\phi_1^{4r_u-1}(\widetilde{\alpha_2})$.
Again, the desired inequality follows from Lemma \ref{lemma::two points partitioned by two geodesics}.
\end{proof}

The multi-twist of a hyperbolic surface that provides a very slight deformation of the hyperbolic structure should be along a set of very short closed geodesics.
We formulate this property in Proposition \ref{proposition::bounded_lmax} and Proposition \ref{proposition::l<=K*epsilon}, which should be well-known.

\begin{proposition}
\label{proposition::bounded_lmax}
Given $\mu\in\Z_{>0}$, there exists a constant $l_{\text{max}}=l_{\text{max}}(h,\mu)$ that depends only on $h$ and $\mu$ such that,
for any multi-twist $T=T_{\alpha_1}^{r_1}\circ\cdots\circ T_{\alpha_m}^{r_m}\in \Mod_h$ along
$\boldsymbol{\alpha}=\{\alpha_1,\ldots,\alpha_m\}$
and $[(X,f_X)]\in\Teich_h$,
if
\begin{equation*}
d_{\Teich}([(X,f_X)],T\cdot [(X,f_X)])\le 2\mu
\end{equation*}
then for all $i=1,\ldots,m$, we get
\begin{equation*}
L_{\alpha_i}([(X,f_X)]) \le l_{\text{max}}.
\end{equation*}
\end{proposition}

\begin{remark}
Technically, there is no need to restrict $\mu$ to being an integer in Proposition \ref{proposition::bounded_lmax}, as well as in Proposition \ref{proposition::l<=K*epsilon}.
However, if $\phi$ is an infinite order peripheral monodromy of a holomorphic map $F:B\rightarrow \mathcal{M}_h$,
then some power $\phi^\mu \eqqcolon T$ being a multi-twist, with $\mu \in \Z_{>0}$.
Consequently, for any marking $[(X,f_X)]$ of the image $F(b)$, where $b\in U \subset B$ lies in the cusp region, the required inequality holds because $l_B(\partial U) = 2$. 
\end{remark}

\begin{proof}
Without loss of generality, we assume that $(X,f_X)$ is a hyperbolic representative of the given $[(X,f_X)]$.
There exists a geodesic pants decomposition $\mathcal{P}_X=\{\gamma_i\}$ of $X$ with each $l_X(\gamma_i)$ bounded above by Bers' constant (cf. Theoem 12.8 in \cite{FarbMargalit2011}).
More precisely, $l_X(\gamma_i)\le 21(h-1)$ for each $\gamma_i\in\mathcal{P}_X$.
Suppose that $\alpha_{X,i}\subset X$ is the (unique) geodesic homotopic to $f_X(\alpha_i)$, for each $\alpha_i\in\boldsymbol{\alpha}$,
and
set $\boldsymbol{\alpha}_X=\{\alpha_{X,i}\mid i=1,\ldots,m\}$.
There are two cases to consider for all $i=1,\ldots,m$.
\\
\textbf{Case 1: $\alpha_{X,i}\in \mathcal{P}_X$.}
Then, $L_{\alpha_i}([(X,f_X)])=l_X(\alpha_{X,i}) \le 21(h-1)$.
\\
\textbf{Case 2: $\alpha_{X,i}\not\in \mathcal{P}_X$.}
Then, there exists a simple closed curve $\gamma\subset \Sigma_h$ such that
$\Int(\boldsymbol{\alpha},\gamma)\ge 1$ and
$L_{\gamma}([(X,f_X)]) \le 21(h-1)$.
By Wolpert's Lemma \ref{lemma::wolpert}, 
\begin{equation*}
L_{\gamma}(T^4\cdot [(X,f_X)])
\le
\exp{\{2\cdot d_{\Teich}([(X,f_X)],T^4\cdot [(X,f_X)])\}}
\cdot L_{\gamma}([(X,f_X)])
\le
e^{16\mu}\cdot L_{\gamma}([(X,f_X)]).
\end{equation*}
Hence, by Lemma \ref{lemma::multi-twist_sum>=l_X(alpha_i)},
we get
\begin{align*}
L_{\alpha_i}([(X,f_X)])
&\le 3\Big(
L_{\gamma}([(X,f_X)]) + L_{\gamma}(T^4\cdot [(X,f_X)])
\Big)\\
&\le
3L_{\gamma}([(X,f_X)])(1+e^{16\mu})
\le 63(h-1)(1+e^{16\mu})\eqqcolon l_{\text{max}}(h,\mu).
\end{align*}
\end{proof}

We shall consider Fenchel-Nielsen coordinates for $\Teich_h$ associated with a pants decomposition $\mathcal{P}$ of $\Sigma_h$.
Let $\mathcal{P}=\{C_i\}$ be a set of closed curves on $\Sigma_h$.
The \emph{length parameter} of $C_i$ is denoted by $l_i=L_{C_i}([(X,f_X)])$.
However, the \emph{twist parameter} $\theta_i$ is chosen to be proportional along $C_i$ so that a positive Dehn twist along $C_i$ changes the
twist parameter by adding $2\pi$.
For $[(X,f_X)],[(X',f_{X'})] \in\Teich_h$, the \emph{Fenchel-Neilsen distance} with respect to $\mathcal{P}$ is defined by
\begin{equation*}
    d_{\text{FN},\mathcal{P}}([(X,f_X)],[(X',f_{X'})])=\sup_i \max\{|\log{l_i}-\log{l_i'}|,|l_i\theta_i-l_i'\theta_i'|\}
\end{equation*}
where $[(X,f_X)],[(X',f_{X'})]$ have Fenchel-Nielsen coordinates $c=(l_i,\theta_i)_i$ and $c'=(l_i',\theta_i')_i$ respectively.
Fenchel-Nielsen distance is introduced and investigated in \cite{ALPSS2011Fenchel}. 
In fact, we have the following quasi-isometric relation between Fenchel-Nielsen distance and Teichm\"uller distance.

\begin{proposition}[Proposition 8.4 and Corollary 8.8 in \cite{ALPSS2011Fenchel}]
\label{proposition::quai-isometry_FN_and_Teich}
Let $[(X,f_X)]$, $[(X',f_{X'})]\in\Teich_h$ be arbitrary and $\mathcal{P}=\{C_i\}$ be a pants decomposition of $\Sigma_h$
such that $[(X,f_X)]$, $[(X',f_{X'})]$ have Fenchel-Nielsen coordinates $c=(l_i,\theta_i)_i$, $c'=(l'_i,\theta'_i)_i$ respectively.
Suppose that, for some constants $N_1,N_2>0$, we have $l_i,l'_i\le N_1$ for all $i=1,2,\ldots$ and $d_{\text{FN},\mathcal{P}}([(X,f_X)],[(X',f_{X'})])\le N_2$.
Then,
\begin{equation*}
d_{\Teich}([(X,f_X)],[(X',f_{X'})])
\le
d_{\text{FN},\mathcal{P}}([(X,f_X)],[(X',f_{X'})])
\cdot N_3(N_1,N_2)
\end{equation*}
and
\begin{equation*}
d_{\text{FN},\mathcal{P}}([(X,f_X)],[(X',f_{X'})])
\le
d_{\Teich}([(X,f_X)],[(X',f_{X'})])
\cdot N_4(N_1,N_2)
\end{equation*}
where the constants  $N_3(N_1,N_2)$ and $N_4(N_1,N_2)$ depend only on $N_1$ and $N_2$.
\end{proposition}

We improve Proposition \ref{proposition::bounded_lmax} using Proposition \ref{proposition::quai-isometry_FN_and_Teich}.

\begin{proposition}
\label{proposition::l<=K*epsilon}
Given $\mu\in\Z_{>0}$,
there exists a constant $K_1=K_1(h,\mu)$ that depends only on $h$ and $\mu$ such that,
for any multi-twist $T=T_{\alpha_1}^{r_1}\circ\cdots\circ T_{\alpha_m}^{r_m}\in \Mod_h$ along
$\boldsymbol{\alpha}=\{\alpha_1,\ldots,\alpha_m\}$
and $[(X,f_X)]\in\Teich_h$,
if
\begin{equation*}
d_{\Teich}([(X,f_X)],T\cdot [(X,f_X)])\le 2\mu
\end{equation*}
then for all $i=1,\ldots,m$, we get
\begin{equation*}
L_{\alpha_i}([(X,f_X)])
\le
K_1 d_{\Teich}([(X,f_X)],T\cdot [(X,f_X)]).
\end{equation*}
\end{proposition}

\begin{proof}
As in Proposition \ref{proposition::bounded_lmax}, we assume that $(X,f_X)$ is a hyperbolic representative of the given $[(X,f_X)]$ and let $\boldsymbol{\alpha}_X$ be the set of geodesics on $X$ homotopic to $f_X(\alpha_i)$ for every $\alpha_i\in\boldsymbol{\alpha}$.
Therefore, there exists a(nother) geodesic pants decomposition $\mathcal{P}_X=\{\gamma_i\}$ of $X$ such that 
\begin{itemize}[-]
\item $\boldsymbol{\alpha}_X\subseteq \mathcal{P}_X$;

\item $l_X(\gamma)\le \Bers(\boldsymbol{\alpha}_X)$ for each $\gamma\in \mathcal{P}_X$,
where $\Bers(\boldsymbol{\alpha}_X)$ is a variation of Bers' constant that depends only on $h$ and lengths of every geodesics in $\boldsymbol{\alpha}_X$,
therefore depends only on $h$ and $\mu$.
\end{itemize}
Set $C_i=f_X^{-1}(\gamma_i)$ and then $\mathcal{P}=\{C_i\}$ is a pants decomposition of $\Sigma$ such that $L_{C_i}([(X,f_X)])\le \Bers(\boldsymbol{\alpha}_X)$, for each $C_i$.
By Wolpert's Lemma \ref{lemma::wolpert}, we further have
\begin{equation*}
L_{C_i}(T\cdot [(X,f_X)]) \le \exp{\{2\cdot d_{\Teich}([(X,f_X)],T\cdot [(X,f_X)]) \}} \cdot L_{C_i}([X,f_X])
\le e^{4\mu} \Bers(\boldsymbol{\alpha}_X).
\end{equation*}
Therefore, by Proposition \ref{proposition::quai-isometry_FN_and_Teich}, there exists a constant $K_1$ depending only on $\Bers(\boldsymbol{\alpha}_X)$ and $e^{4\mu}\Bers(\boldsymbol{\alpha}_X)$ such that
\begin{equation*}
L_{\alpha_i}([(X,f_X)])
\le
\frac{1}{2 \pi}
d_{\text{FN},\mathcal{P}}([(X,f_X)],T\cdot [(X,f_X)])
\le
K_1 d_{\Teich}([(X,f_X)],T\cdot [(X,f_X)])
\end{equation*}
for all $i=1,\ldots,m$.
\end{proof}

From now on, we consider a hyperbolic cusp region $U$ and a non-constant map $F:U\rightarrow\mathcal{M}_h$ that is distance-decreasing for $(1/2) d_U$ on $U$ and $d_{\mathcal{M}}$ on $\mathcal{M}_h$.

Suppose that $U=\langle g \rangle \backslash \mathcal{B}$ with $\mathcal{B}\subset\Hyperbolic^2$ a horoball and $g\in \Aut(\Hyperbolic^2)$ a parabolic isometry.
Therefore, the map $F$ can be lifted to a map $\widetilde{F}:\mathcal{B}\rightarrow \Teich_h$.
Let $\phi\in \Mod_h$ be such that, given a generating loop $\gamma\subset U$ based at $p\in U$ and a lift $\widetilde{p}\in \mathcal{B}$ of $p$, one can lift $F(\gamma)$ to a path joining $\widetilde{F}(\widetilde{p})$ to $\phi \cdot \widetilde{F}(\widetilde{p})$.
This mapping class $\phi$ should satisfy the inequality 
$\epsilon / 2 \ge d_{\Teich}(\widetilde{F}(\widetilde{p}), \phi \cdot \widetilde{F}(\widetilde{p}))$,
where $0<\epsilon\le 2$ and the horocycle $H_{\epsilon}\subset U$ of length $\epsilon$ contains $p$.
Such a mapping class is called the monodromy of $F$.

We suppose that a power $\phi^{\mu}$ is exactly a multi-twist $T$ along $\boldsymbol{\alpha}$.
Each point $p\in U$ determines the unique horocycle $H_{\varepsilon}\subset U$ such that $p\in H_{\varepsilon}$.
Each lift $\widetilde{p}\in \mathcal{B}$ of $p$ determines the 
geodesic length $l=L_{\alpha_1}(\widetilde{F}(\widetilde{p}))$.
We associate the length amount $\epsilon$ and the length amount $l$ to show that $F$ is a quasi-isometric embedding.

\begin{theorem}
\label{theorem::cusp_region->end is quasi-isometric}
Given $\epsilon>0$,
there exists
$K_2=K_2(h,\mu,\epsilon)$, depending only on $h$, $\mu$ and $\epsilon$, that satisfies the following:
Suppose that $\sys(F(p_{\text{max}}))\ge \epsilon$ for some $p_{\text{max}}\in\partial U$.
Then, for any pair of points $(p_1,p_2)$ in $U$,
we have
\begin{equation*}
\frac{1}{2} d_U(p_1,p_2)\ge
d_{\mathcal{M}}(F(p_1),F(p_2))\ge
\frac{1}{2} d_U(p_1,p_2)-K_2.
\end{equation*}
\end{theorem}

\begin{proof}
Let $p_{\max}\in\partial U$ and $p\in H_{\varepsilon}\subset U$ be arbitrary with $\varepsilon\le 2$.
Take a lift $\widetilde{p_{\max}}\in \mathcal{B}$ of $p_{\max}$ and a lift $\widetilde{p}\in \mathcal{B}$ of $p$ such that $d_{\mathcal{M}}(F(p_{\max}), F(p))=d_{\Teich}(\phi_{\max} \cdot \widetilde{F}(\widetilde{p_{\max}}), \phi \cdot \widetilde{F}(\widetilde{p}))$,
for some $\phi_{\max}, \phi \in \Mod_h$.
For convenience, we set $q_{\max}=F(p_{\max})$, $q=F(p)$, $\widetilde{q_{\max}}=\widetilde{F}(\widetilde{p_{\max}})$ and $\widetilde{q}=\widetilde{F}(\widetilde{p})$.
By Wolpert's Lemma \ref{lemma::wolpert}, Proposition \ref{proposition::l<=K*epsilon} and the triangle inequality in $(U,d_U)$, we have
\begin{align*}
d_{\mathcal{M}}(q,q_{\max})
&=d_{\Teich}(\widetilde{q},\phi^{-1}\circ \phi_{\max}\cdot \widetilde{q_{\max}})
\ge
\frac{1}{2}\log{
  \frac{L_{\alpha_1}(\phi^{-1}\circ \phi_{\max}\cdot \widetilde{q_{\max}})}{L_{\alpha_1}(\widetilde{q})}
}\\
&\ge
\frac{1}{2}\log{
  \frac{\sys(\widetilde{q_{\max}})}{K_1(h,\mu) d_{\Teich}(\widetilde{q}, T\cdot \widetilde{q})}
}
\ge
\frac{1}{2}\log{
  \frac{\sys(q_{\max})}{K_1(h,\mu) \mu \epsilon / 2}
}
=
\frac{1}{2}\Big\{\log{
  \frac{\sys(q_{\max})}{K_1(h,\mu) \mu}
}
+
\log{\frac{2}{\epsilon}}\Big\}\\
&\ge
\frac{1}{2}d_U(p,p_{\max})-K'_2
\end{align*}
where $K'_2=1-(1/2)\log{\sys(q_{\max})}+(1/2)\log{K_1(h,\mu) \mu}$.

In general, let $p_1,p_2\in U$ be arbitrary.
Set $q_1=F(p_1)$, $q_2=F(p_2)$ and
take the corresponding horocycles
$H_{\varepsilon_1}\ni p_1$, $H_{\varepsilon_2}\ni p_2$.
Using the above inequality and triangle inequalities in both $(U,d_U)$ and $(\mathcal{M}_h,d_{\mathcal{M}})$, we conclude that
\begin{align*}
\frac{1}{2}d_U(p_1,p_2)
&\ge
d_{\mathcal{M}}(q_1,q_2)
\ge
|d_{\mathcal{M}}(q_1,q_{\max})-d_{\mathcal{M}}(q_2,q_{\max})|
\ge
\frac{1}{2}|d_U(p_1,p_{\max})-d_U(p_2,p_{\max})|-K'_2\\
&\ge
\frac{1}{2} \max\Big\{
  \Big(\log{\frac{2}{\varepsilon_1}}-2\Big)-
  \Big(\log{\frac{2}{\varepsilon_2}}+2\Big),
  \Big(\log{\frac{2}{\varepsilon_2}}-2\Big)-
  \Big(\log{\frac{2}{\varepsilon_1}}+2\Big)
\Big\}-K'_2\\
&=
\frac{1}{2} \max\Big\{
  \log{\frac{\varepsilon_1}{\varepsilon_2}},
  \log{\frac{\varepsilon_2}{\varepsilon_1}}
\Big\}-2-K'_2\\
&\ge
\frac{1}{2}(d_U(p_1,p_2)-2)-2-K'_2
=
\frac{1}{2}d_U(p_1,p_2)-3-K'_2.
\end{align*}
\end{proof}
\subsection{Proof of Theorem \ref{thmx::quasi-isometry}}
\label{subsection::Theorem A}


Consider an oriented hyperbolic surface $B$ of type $(g,n)$, which has $n$ cusps.
Let $U_i$ be the cusp region of the $i$-th cusp, which is of area $2$ and bounded by a horocycle of length $2$, for $i=1,\ldots, n$.
The complement is a compact hyperbolic surface with boundary, denoted by $B_{cp}\subset B$.

The proof of Theorem \ref{thmx::quasi-isometry} - (\RNum{1}) is made up of two lemmata.
The first lemma claims that the holomorphic map restricted to a cusp region $U_i$ is a quasi-isometric embedding whose parameters depend not only on $(g,n)$, $h$ and $\sys(B)$ but also on $\sys(F(b))$ for an arbitrary point $b\in B_{cp}$.
The second lemma claims that $\sys(F(b))$ is bounded uniformly for $b\in B_{cp}$.
Theorem \ref{thmx::quasi-isometry} - (\RNum{2}) is a consequence of Theorem \ref{thmx::quasi-isometry} - (\RNum{1}) due to the fact that $\diam(\overline{B_{cp}})$ has an upper bound based on $\sys(B)$.

\begin{lemma}
\label{lemma::quasi-isometric embedding}
Given $\epsilon>0$,
there exists a constant $K_3=K_3(g,n,h,\epsilon)$, depending only on $(g,n)$, $h$ and $\epsilon$, such that the following holds:
Let $B$ be an oriented hyperbolic surface of type $(g,n)$ and let $F:B\rightarrow \mathcal{M}_h$ be a non-constant holomorphic map with a monodromy homomorphism $F_*\in \Hom(\pi_1(B,t),\Mod_h)$
such that the peripheral monodromy of the $i$-th cusp is of infinite order, for some $i=1,\ldots,n$.
Suppose that $\sys(B)\ge \epsilon$ and
$\sys(F(b))\ge \epsilon$ for some $b\in B_{cp}$.
Then, for each pair of points $(p_1,p_2)\in U_i\times U_i$, we have
\begin{equation*}
\frac{1}{2}d_B(p_1,p_2)\ge
d_{\mathcal{M}}(F(p_1),F(p_2))\ge
\frac{1}{2}d_B(p_1,p_2)-K_3.
\end{equation*}
\end{lemma}

\begin{proof}
Since $F$ is holomorphic, it is automatically distance-decreasing for $(1/2) d_B$ on $B$ and $d_{\mathcal{M}}$ on $\mathcal{M}_h$.
Now regard the hyperbolic surface $B$ as the union of the compact region $B_{cp}$ and $n$ more disjoint cusp regions $U_1,\ldots,U_n$ each bounded by a horocycle of length $2$.
Select an arbitrary boundary point $p_{\max,i}\in \partial U_i$ and
let $\diam{(B_{cp})}$ be the diameter of $B_{cp}$.

A peripheral monodromy of the $i$-th cusp, denoted by $\phi$, is reducible, of infinite order and has no pseudo-Anosov reduced component.
Therefore,
some power $\phi^{\mu}$ is identity on each component, where $\mu$ is bounded above by a constant determined by $h$.
Hence, $\phi^{\mu}$ is a multi-twist.

Suppose that $\sys(F(b))\ge \epsilon$ for some $b\in B_{cp}$.
Let $(p_1,p_2)$ be a pair of points in $U_i$.
By Theorem \ref{theorem::cusp_region->end is quasi-isometric}, the difference between
$d_{\mathcal{M}}(F(p_1),F(p_2))$
and
$(1/2)d_B(p_1,p_2)$
is bounded by a constant depending only on $h,\mu$ and a lower bound of $\sys(F(p_{\max,i}))$.
By Mumford's compactness,
the diameter $\diam{(B_{cp})}$ is bounded by a constant determined by
$(g,n)$ and $\epsilon$.
Wolpert's Lemma \ref{lemma::wolpert} then shows a mutual dependence between $\sys(F(b))$ and $\sys(F(p_{\max,i}))$.
Hence, the unique parameter used in the desired inequality depends only on $(g,n)$, $h$ and $\epsilon$.
\end{proof}

\begin{lemma}
\label{lemma::bounded systole for compact part}
Given $\epsilon>0$,
there exists a constant $K_4=K_4(g,n,h,\epsilon)$, depending only on $(g,n)$, $h$ and $\epsilon$, such that the following holds:
Let $B$ be an oriented hyperbolic surface of type $(g,n)$ such that $\sys(B)\ge \epsilon$
and let $F:B\rightarrow \mathcal{M}_h$ be a non-constant holomorphic map.
Then, for any $p\in B_{cp}$, we have
\begin{equation*}
    \sys(F(p)) \ge K_4.
\end{equation*}
\end{lemma}

\begin{proof}
One can derive this from Theorem \ref{theorem::Ct_belongs_to_compactK} and Wolpert's Lemma \ref{lemma::wolpert}.
\end{proof}
\subsection{Quasi-isometrically embedded fundamental polygon}
\label{subsection::Theorem B}

As in Theorem \ref{thmx::quasi-isometry-fundamental-domain-n=0} from the introduction,
we provide a necessary condition for the existence of quasi-isometrically embedded fundamental domains in this subsection, for the general case.

\begin{theorem}
\label{theorem::quasi-isometry-fundamental-domain}
Let $(g,n)$, $h$ and $\epsilon$ satisfy $2g-2+n>0$, $h\ge 2$ and $\epsilon >0$.
There exists a constant $K=K(g,n,h,\epsilon)$, depending only on $g,n,h$ and $\epsilon$, such that the following holds.
Let $B=\Gamma \backslash \Hyperbolic^2$ be an oriented hyperbolic surface of type $(g,n)$ with $\sys(B)\ge \epsilon$
and cusp regions $U_1,\ldots, U_n\subset B$.
Let $D\subset \Hyperbolic^2$ be a convex fundamental polygon for $B$, having exactly $n$ ideal points.
Suppose that
$F:B\rightarrow\mathcal{M}_h$ is a non-constant holomorphic map with a monodromy homomorphism $F_*\in \Hom(\pi_1(B,t),\Mod_h)$.
Assume the following conditions:
\begin{itemize}[-]
\item The peripheral monodromies of all cusps are of infinite order.

\item For every pair $i\neq j$ with $i,j \in \{1,\ldots,n\}$, there exists a geodesic segment $\kappa_{i,j}\subset B$ joining $\partial U_i$ to $\partial U_j$, with a lift $\widetilde{\kappa_{i,j}} \subset D$,
such that peripheral monodromies of $U_i$ and $U_j$ are not disjointed along $\kappa_{i,j}$,
\end{itemize}
Then, the restriction
$\left.\widetilde{F}\right|_{D}: (D,(1/2)d_{\Hyperbolic^2}) \rightarrow (\Teich_h,d_{\Teich})$ is a $(2,K+\diam(D_{\text{cp}}))$-quasi-isometric embedding,
where $D_{\text{cp}}$ is the compact part of $D$ obtained by removing the lift of all cusp regions.
\end{theorem}

In the case $(g,n)=(0,n)$, suppose that $(\phi_1,\ldots,\phi_n)$ is a global monodromy,
then there exists a fundamental polygon $D\subset \Hyperbolic^2$ of $B$ and segments $\kappa_{i,j}\subset B$ together with lifts $\widetilde{\kappa_{i,j}}\subset D$ that satisfy the following equivalence:
the peripheral monodromies of $U_i$ and $U_j$ are disjointed along $\kappa_{i,j}$ if and only if $\phi_i$ and $\phi_j$ are disjointed.
On the other hand, a different choice of $\gamma_1, \ldots, \gamma_n$ changes $(\phi_1,\ldots,\phi_n)$ by a sequence of Hurwitz moves:
\begin{equation*}
(\ldots,\phi_{i}\circ\phi_{i+1}\circ\phi_i^{-1},\phi_i,\ldots)
\xmapsfrom{L_i}
(\ldots,\phi_i,\phi_{i+1},\ldots)
\xmapsto{R_i}
(\ldots,\phi_{i+1},\phi_{i+1}^{-1}\circ\phi_i\circ\phi_{i+1},\ldots).
\end{equation*}
Thereby, the existence of a fundamental polygon $D$ as in Theorem \ref{theorem::quasi-isometry-fundamental-domain} such that $\left.\widetilde{F}\right|_{D}$ is a quasi-isometric embedding is equivalent to
the existence of a tuple being Hurwitz equivalent to $(\phi_1,\ldots,\phi_n)$ whose components are of infinite order and pairwise non-disjointed, as in Theorem \ref{thmx::quasi-isometry-fundamental-domain-n=0}.

Theorem \ref{theorem::quasi-isometry-fundamental-domain} comes from the following lemma.

\begin{lemma}
Given $\epsilon>0$,
there exists a constant $K_5=K_5(g,n,h,\epsilon)$, depending only on $(g,n)$, $h$ and $\epsilon$, such that the following holds:
Let $B$ be an oriented hyperbolic surface of type $(g,n)$ with $\sys(B)\ge \epsilon$
and
let $D\subset \Hyperbolic^2$ be a convex fundamental polygon of $B$ with exactly $n$ ideal points.
Let $F:B\rightarrow \mathcal{M}_h$ be a non-constant holomorphic map with a monodromy homomorphism $F_*\in \Hom(\pi_1(B,t), \Mod_h)$.
For some distinct $i\neq j$, where $i,j \in \{1,\ldots,n\}$,
if peripheral monodromies of the $i$-th and the $j$-th cusps are of infinite order and are not disjointed along some geodesic segment $\kappa_{i,j}\subset B$ with a lift $\widetilde{\kappa_{i,j}}\subset D$,
then for each pair of points $(p_1,p_2)\in U_i\times U_j$,
where $\widetilde{p_1}\in D$ and $\widetilde{p_2}\in D$ are lifts of $p_1$ and $p_2$,
we have the inequality
\begin{equation*}
\frac{1}{2}d_{\Hyperbolic^2}(\widetilde{p_1},\widetilde{p_2})\ge
d_{\Teich}(\widetilde{F}(\widetilde{p_1}),\widetilde{F}(\widetilde{p_2}))\ge
\frac{1}{4}d_{\Hyperbolic^2}(\widetilde{p_1},\widetilde{p_2})-K_5-l_B(\kappa_{i,j}).
\end{equation*}
\end{lemma}

\begin{proof}
Without loss of generality, we assume that $i=1$ and $j=2$.
Suppose that $p_1\in H_{\epsilon_1}\subset U_1$
where $H_{\epsilon_1}$ is the horocycle of length $0<\epsilon_1\le 2$ within the cusp region $U_1$.
Suppose that $p_2\in H_{\epsilon_2}\subset U_2$
where $H_{\epsilon_2}$ is the horocycle of length $0<\epsilon_2\le 2$ within the cusp region $U_2$.

Consider the peripheral monodromies $\phi_1$ and $\phi_2$ associated to $\kappa$.
Therefore, some power $\phi_1^{\mu_1}$ is a multi-twist along a multi-curve $\boldsymbol{\alpha}_1$ and some power $\phi_2^{\mu_2}$ is a multi-twist along a multi-curve $\boldsymbol{\alpha}_2$, where both $\mu_1$ and $\mu_2$ are bounded above by a constant determined by $h$.
There exist $\alpha_1\in\boldsymbol{\alpha}_1$ and $\alpha_2\in\boldsymbol{\alpha}_2$ such that $\Int(\alpha_1,\alpha_2)\ge 1$.

Take $q_1=F(p_1)$, $q_2=F(p_2)\in\mathcal{M}_h$ and
set $\widetilde{q_1}=\widetilde{F}(\widetilde{p_1})$, $\widetilde{q_2}=\widetilde{F}(\widetilde{p_2})\in\Teich_h$.
By Proposition \ref{proposition::l<=K*epsilon},
since
$d_{\Teich}(\widetilde{q_1},\phi_1^{\mu_1}\cdot \widetilde{q_1})
\le \mu_1 \cdot \epsilon_1
\le 2\mu_1$ and
$d_{\Teich}(\widetilde{q_2},\phi_2^{\mu_2}\cdot \widetilde{q_2})
\le \mu_2 \cdot \epsilon_2
\le 2\mu_2$
we have
\begin{align*}
L_{\alpha_1}(\widetilde{q_1})
\le K_1(h,\mu_1) \cdot d_{\Teich}(\widetilde{q_1},\phi_1^{\mu_1}\cdot \widetilde{q_1})
\le K_1(h,\mu_1)\cdot \mu_1 \cdot \epsilon_1,\\
L_{\alpha_2}(\widetilde{q_2})
\le K_1(h,\mu_2) \cdot d_{\Teich}(\widetilde{q_2},\phi_2^{\mu_2}\cdot \widetilde{q_2})
\le K_1(h,\mu_2)\cdot \mu_2 \cdot \epsilon_2.
\end{align*}
Besides, since $\Int(\alpha_1,\alpha_2)\ge 1$, we have
\begin{equation*}
\sinh{\Big(\frac{L_{\alpha_1}(\widetilde{q_1})}{2}\Big)}
\sinh{\Big(\frac{L_{\alpha_2}(\widetilde{q_1})}{2}\Big)}
\ge 1
\text{~~and~~}
\sinh{\Big(\frac{L_{\alpha_1}(\widetilde{q_2})}{2}\Big)}
\sinh{\Big(\frac{L_{\alpha_2}(\widetilde{q_2})}{2}\Big)}
\ge 1
\end{equation*}
which implies that
$L_{\alpha_1}(\widetilde{q_2})\ge 2\arcsinh{\frac{1}{\sinh{K_1(h,\mu_2)\cdot \mu_2}}}$
and
$L_{\alpha_2}(\widetilde{q_1})\ge 2\arcsinh{\frac{1}{\sinh{K_1(h,\mu_1)\cdot \mu_1}}}$.
Therefore, by Wolpert's Lemma \ref{lemma::wolpert}, we have
\begin{align*}
d_{\Hyperbolic^2}(\widetilde{p_1},\widetilde{p_2})
&\ge 2d_{\Teich}(\widetilde{F}(\widetilde{p_1}),\widetilde{F}(\widetilde{p_2}))
= 2d_{\Teich}(\widetilde{q_1},\widetilde{q_2})
\ge \frac{1}{2}
\log{\frac{L_{\alpha_1(\widetilde{q_2})}}{L_{\alpha_1(\widetilde{q_1})}}}
+ \frac{1}{2}
\log{\frac{L_{\alpha_2(\widetilde{q_1})}}{L_{\alpha_2(\widetilde{q_2})}}}\\
&\ge 
\frac{1}{2}\big(
\log{\frac{2}{\epsilon_1}}+\log{\frac{2}{\epsilon_2}}
\big)
-K_{5,1,2}(h,\mu_1,\mu_2)
-K_{5,2,1}(h,\mu_1,\mu_2)
\end{align*}
where
\begin{align*}
K_{5,1,2}(h,\mu_1,\mu_2)=
\frac{1}{2}
\log{
\frac{
K_1(h,\mu_1)\cdot \mu_1
}{
\arcsinh{
\frac{1}
{
\sinh{K_1(h,\mu_2)\cdot\mu_2}
}
}
}
},
\\
K_{5,2,1}(h,\mu_1,\mu_2)=
\frac{1}{2}
\log{
\frac{
K_1(h,\mu_2)\cdot \mu_2
}{
\arcsinh{
\frac{1}
{
\sinh{K_1(h,\mu_1)\cdot\mu_1}
}
}
}
}.
\end{align*}
Using triangle inequality in $(\Hyperbolic^2,d_{\Hyperbolic^2})$ and the fact that $D$ is convex and bounded by geodesic segments, we conclude that
\begin{align*}
d_{\Hyperbolic^2}(\widetilde{p_1},\widetilde{p_2})
&\ge 2d_{\Teich}(\widetilde{q_1},\widetilde{q_2})
\ge \frac{1}{2}d_{\Hyperbolic^2}(\widetilde{p_1},\widetilde{p_2})
- \frac{1}{2} \diam(\widetilde{B_{cp}}\cap D)
- K_{5,1,2}(h,\mu_1,\mu_2)
- K_{5,2,1}(h,\mu_1,\mu_2)
\end{align*}
where $\widetilde{B_{cp}}\subset \Hyperbolic^2$ is the lift of $B_{cp}\subset B$.
\end{proof}
\section{Examples and applications}
\label{section::example}

This section is intended to provide several examples, remarks concerning and consequences of Theorem \ref{thmx::quasi-isometry} and Theorem \ref{thmx::quasi-isometry-fundamental-domain-n=0}.
We focus on holomorphic curves in $\mathcal{M}_2$.
Let $\gamma_1,\gamma_2,\gamma_3,\gamma_4,\gamma_5$ and $\tau,\sigma$ be closed curves on $\Sigma_2$ represented in Figure \ref{figure::genus2-5Dehntwists} and Figure \ref{figure::genus2-2Dehntwists}.
It is well-known (cf. \cite[Theorem 4.8]{Birman1975} and \cite[Figure 1]{auroux2003genus2}) that $\Mod_2$ is generated by the five Dehn twists $T_{\gamma_1}, T_{\gamma_2}, T_{\gamma_3}, T_{\gamma_4}, T_{\gamma_5}$.
One can check that
$I = (T_{\gamma_1}\circ T_{\gamma_2}\circ T_{\gamma_3}\circ T_{\gamma_4})^5$.

\begin{figure}[htb]
\centering
\begin{minipage}{.45\textwidth}
  \centering
  \def\svgwidth{.5\textwidth}
\begingroup%
  \makeatletter%
  \providecommand\color[2][]{%
    \errmessage{(Inkscape) Color is used for the text in Inkscape, but the package 'color.sty' is not loaded}%
    \renewcommand\color[2][]{}%
  }%
  \providecommand\transparent[1]{%
    \errmessage{(Inkscape) Transparency is used (non-zero) for the text in Inkscape, but the package 'transparent.sty' is not loaded}%
    \renewcommand\transparent[1]{}%
  }%
  \providecommand\rotatebox[2]{#2}%
  \newcommand*\fsize{\dimexpr\f@size pt\relax}%
  \newcommand*\lineheight[1]{\fontsize{\fsize}{#1\fsize}\selectfont}%
  \ifx\svgwidth\undefined%
    \setlength{\unitlength}{566.92913386bp}%
    \ifx\svgscale\undefined%
      \relax%
    \else%
      \setlength{\unitlength}{\unitlength * \real{\svgscale}}%
    \fi%
  \else%
    \setlength{\unitlength}{\svgwidth}%
  \fi%
  \global\let\svgwidth\undefined%
  \global\let\svgscale\undefined%
  \makeatother%
  \begin{picture}(1,1.5)%
    \lineheight{1}%
    \setlength\tabcolsep{0pt}%
    \put(0,0){\includegraphics[width=\unitlength,page=1]{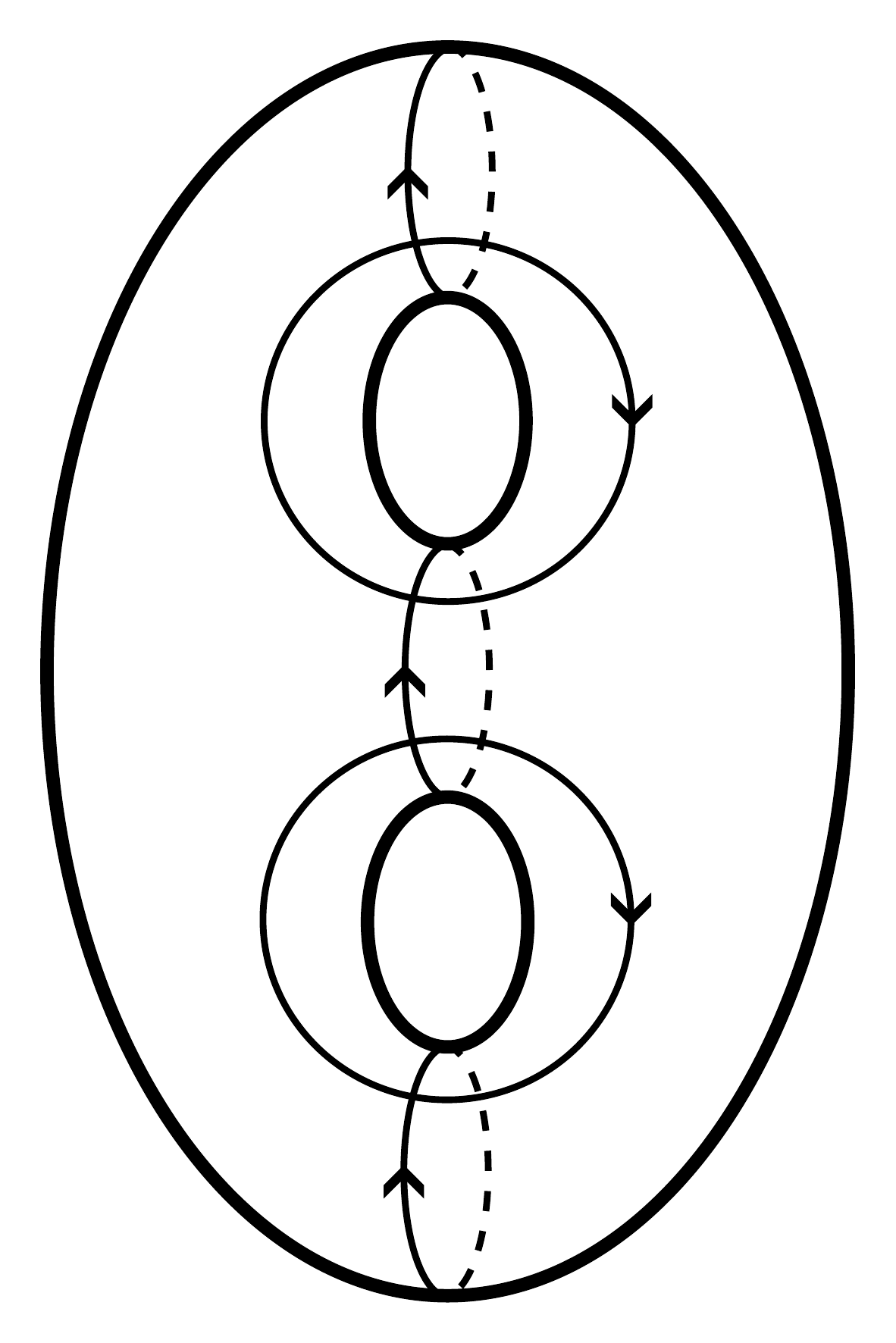}}%
    \put(0.33912133,1.30618634){\color[rgb]{0,0,0}\makebox(0,0)[lt]{\lineheight{1.25}\smash{\begin{tabular}[t]{l}$\gamma_1$\end{tabular}}}}%
    \put(0.18251001,1.02474879){\color[rgb]{0,0,0}\makebox(0,0)[lt]{\lineheight{1.25}\smash{\begin{tabular}[t]{l}$\gamma_2$\end{tabular}}}}%
    \put(0.32759882,0.73706686){\color[rgb]{0,0,0}\makebox(0,0)[lt]{\lineheight{1.25}\smash{\begin{tabular}[t]{l}$\gamma_3$\end{tabular}}}}%
    \put(0.1889703,0.44508915){\color[rgb]{0,0,0}\makebox(0,0)[lt]{\lineheight{1.25}\smash{\begin{tabular}[t]{l}$\gamma_4$\end{tabular}}}}%
    \put(0.32544264,0.16570674){\color[rgb]{0,0,0}\makebox(0,0)[lt]{\lineheight{1.25}\smash{\begin{tabular}[t]{l}$\gamma_5$\end{tabular}}}}%
  \end{picture}%
\endgroup%

  \caption{Five closed curves on $\Sigma_2$ along which the Dehn twists generate $\Mod_2$.
  We have chosen their orientations for later use.}
  \label{figure::genus2-5Dehntwists}
\end{minipage}%
\hfill
\begin{minipage}{.45\textwidth}
  \centering
  \def\svgwidth{.5\textwidth}
\begingroup%
  \makeatletter%
  \providecommand\color[2][]{%
    \errmessage{(Inkscape) Color is used for the text in Inkscape, but the package 'color.sty' is not loaded}%
    \renewcommand\color[2][]{}%
  }%
  \providecommand\transparent[1]{%
    \errmessage{(Inkscape) Transparency is used (non-zero) for the text in Inkscape, but the package 'transparent.sty' is not loaded}%
    \renewcommand\transparent[1]{}%
  }%
  \providecommand\rotatebox[2]{#2}%
  \newcommand*\fsize{\dimexpr\f@size pt\relax}%
  \newcommand*\lineheight[1]{\fontsize{\fsize}{#1\fsize}\selectfont}%
  \ifx\svgwidth\undefined%
    \setlength{\unitlength}{566.92913386bp}%
    \ifx\svgscale\undefined%
      \relax%
    \else%
      \setlength{\unitlength}{\unitlength * \real{\svgscale}}%
    \fi%
  \else%
    \setlength{\unitlength}{\svgwidth}%
  \fi%
  \global\let\svgwidth\undefined%
  \global\let\svgscale\undefined%
  \makeatother%
  \begin{picture}(1,1.5)%
    \lineheight{1}%
    \setlength\tabcolsep{0pt}%
    \put(0,0){\includegraphics[width=\unitlength,page=1]{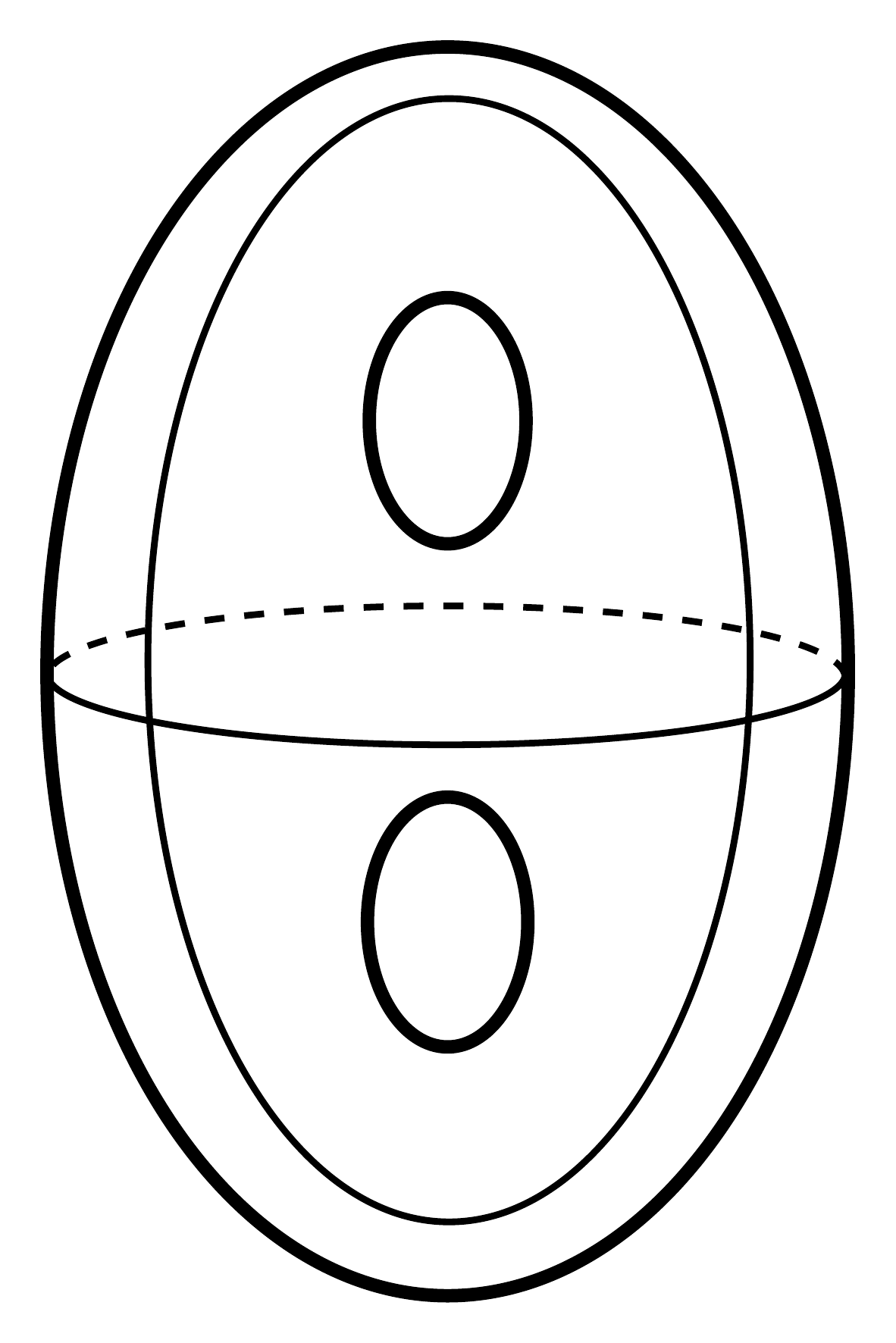}}%
    \put(0.47205299,0.70439407){\color[rgb]{0,0,0}\makebox(0,0)[lt]{\lineheight{1.25}\smash{\begin{tabular}[t]{l}$\sigma$\end{tabular}}}}%
    \put(0.47201294,1.30094854){\color[rgb]{0,0,0}\makebox(0,0)[lt]{\lineheight{1.25}\smash{\begin{tabular}[t]{l}$\tau$\end{tabular}}}}%
  \end{picture}%
\endgroup%

  \caption{Another pair of closed curves on $\Sigma_2$.}
  \label{figure::genus2-2Dehntwists}
\end{minipage}
\end{figure}

\subsection{Quasi-isometrically but non-isometrically immersed curves}

Theorem \ref{thmx::quasi-isometry} provides a sufficient condition on the monodromy homomorphism for a holomorphic map $F:B\rightarrow \mathcal{M}_h$ to be a quasi-isometric immersion.
On the other hand, the monodromy homomorphism of an isometric immersion $F:B\rightarrow \mathcal{M}_h$ is essentially purely pseudo-Anosov (see Definition \ref{definition::eppA} and Theorem \ref{theorem::Teichmuller curve is eppA}).
Therefore, we present a criterion for a holomorphic curve to be quasi-isometrically but not isometrically immersed.

\begin{criterion}
Let $B$ be an oriented hyperbolic surface of type $(g,n)$ and $F:B\rightarrow \mathcal{M}_h$ be a non-constant holomorphic map. Suppose that
(a) the monodromy homomorphism is not essentially purely pseudo-Anosov, 
(b) all peripheral monodromies are of infinite order.
Then, the holomorphic curve $F(B)\subset \mathcal{M}_h$ is quasi-isometrically but not isometrically immersed.
\end{criterion}

In the rest of this subsection, we construct a quasi-isometrically but not isometrically immersed holomorphic curve $F(B)\subset \mathcal{M}_2$ of type $(0,6)$.
In addition, there exists a desired fundamental polygon $D$ as in Theorem \ref{thmx::quasi-isometry-fundamental-domain-n=0} such that
$\left.\widetilde{F}\right|_{D}$ is a quasi-isometric embedding.

\begin{example}
\label{example::quasi-example}
Set $B=\C \setminus \{-2,-1,0,1,2\}$ and
\begin{equation*}
C'=\left\{
\big(
[X_0:X_1],[Y_0:Y_1],b
\big)
\in \C P^1 \times \C P^1\times B
\middle|
\begin{aligned}
X_0^6 Y_1^2 =&
(X_1+X_0 b)
(X_1-X_0 b)
(X_1 + X_0)\\
&(X_1 + 2X_0)
(X_1 - X_0)
(X_1 - 2X_0)
Y_0^2    
\end{aligned}
\right\}.
\end{equation*}
Let $\pi':C'\rightarrow B$ be a holomorphic map with $\pi'([X_0:X_1],[Y_0:Y_1],t)=t$.
Then, each fibre $\pi'^{-1}(b)$ is a double cover of $\C P^1$ via
$\pi'^{-1}(b) \ni ([X_0:X_1],[Y_0:Y_1],t)\mapsto [X_0:X_1]\in \C P^1$
with branch points
\begin{equation*}
P_1=[1:b], P_2=[1:2], P_3=[1:1], P_4=[1:-1], P_5=[1:-2], P_6=[1:-b]
\text{ and }
\infty=[0:1].
\end{equation*}
Therefore, the resolution at $[0:1]$ for every $t\in B$ is a holomorphic family $C/B$ of Riemann surfaces of genus $2$, say $\pi:C\rightarrow B$, which is non-isotrivial.
\end{example}

\begin{proposition}
\label{proposition::example1}
The classifying map of the holomorphic family $C/B$ in Example \ref{example::quasi-example} is a quasi-isometric but not-isometric immersion.
Moreover, the lift of the classifying map restricted to some fundamental polygon is a quasi-isometric embedding. 
\end{proposition}

\begin{proof}
The base $B$ is a Riemann surface of type $(0,6)$.
To illustrate the monodromy homomorphism, 
we fix the base point $t\coloneqq 3\in B$ and investigate generic fibres at
$b\in \Gamma\subset B$
where $\Gamma$ is shown in
Figure \ref{figure::generic-positions}.
In fact, the resolution of $\pi'^{-1}(t)$ is the union of two copies of a single-valued branch that are glued along the boundary, where the boundary consists of three connected components.
Figure \ref{figure::single-valued-branch}
shows a piecewise correspondence between the algebraic curve $\pi^{-1}(t)$ and the topological surface $\Sigma_2$.

\begin{figure}[htb]
\centering
\begin{minipage}{.3\textwidth}
  \centering
  \def\svgwidth{.9\textwidth}
\begingroup%
  \makeatletter%
  \providecommand\color[2][]{%
    \errmessage{(Inkscape) Color is used for the text in Inkscape, but the package 'color.sty' is not loaded}%
    \renewcommand\color[2][]{}%
  }%
  \providecommand\transparent[1]{%
    \errmessage{(Inkscape) Transparency is used (non-zero) for the text in Inkscape, but the package 'transparent.sty' is not loaded}%
    \renewcommand\transparent[1]{}%
  }%
  \providecommand\rotatebox[2]{#2}%
  \newcommand*\fsize{\dimexpr\f@size pt\relax}%
  \newcommand*\lineheight[1]{\fontsize{\fsize}{#1\fsize}\selectfont}%
  \ifx\svgwidth\undefined%
    \setlength{\unitlength}{992.12598425bp}%
    \ifx\svgscale\undefined%
      \relax%
    \else%
      \setlength{\unitlength}{\unitlength * \real{\svgscale}}%
    \fi%
  \else%
    \setlength{\unitlength}{\svgwidth}%
  \fi%
  \global\let\svgwidth\undefined%
  \global\let\svgscale\undefined%
  \makeatother%
  \begin{picture}(1,0.28571429)%
    \lineheight{1}%
    \setlength\tabcolsep{0pt}%
    \put(0,0){\includegraphics[width=\unitlength,page=1]{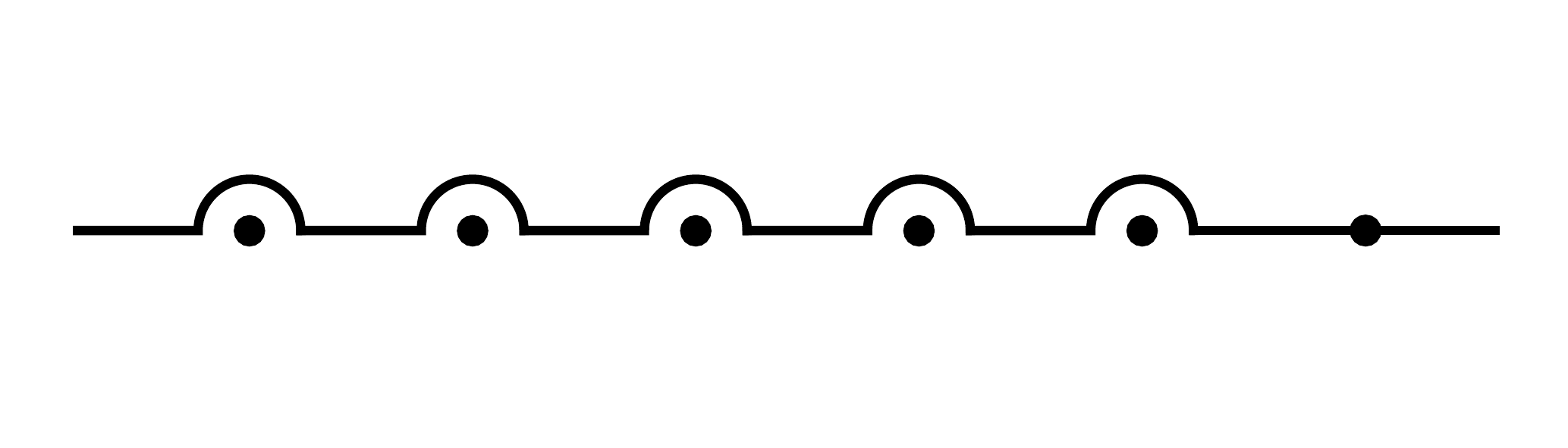}}%
    \put(0.12373876,0.06012466){\color[rgb]{0,0,0}\makebox(0,0)[lt]{\lineheight{1.25}\smash{\begin{tabular}[t]{l}$-2$\end{tabular}}}}%
    \put(0.26607627,0.06012466){\color[rgb]{0,0,0}\makebox(0,0)[lt]{\lineheight{1.25}\smash{\begin{tabular}[t]{l}$-1$\end{tabular}}}}%
    \put(0.42806859,0.0624784){\color[rgb]{0,0,0}\makebox(0,0)[lt]{\lineheight{1.25}\smash{\begin{tabular}[t]{l}$0$\end{tabular}}}}%
    \put(0.5704061,0.0624784){\color[rgb]{0,0,0}\makebox(0,0)[lt]{\lineheight{1.25}\smash{\begin{tabular}[t]{l}$1$\end{tabular}}}}%
    \put(0.71274362,0.06608407){\color[rgb]{0,0,0}\makebox(0,0)[lt]{\lineheight{1.25}\smash{\begin{tabular}[t]{l}$2$\end{tabular}}}}%
    \put(0.83542603,0.06970867){\color[rgb]{0,0,0}\makebox(0,0)[lt]{\lineheight{1.25}\smash{\begin{tabular}[t]{l}$t=3$\end{tabular}}}}%
    \put(0.75709935,0.22213488){\color[rgb]{0,0,0}\makebox(0,0)[lt]{\lineheight{1.25}\smash{\begin{tabular}[t]{l}$\Gamma\subset B$\end{tabular}}}}%
  \end{picture}%
\endgroup%

  \caption{The subset $\Gamma\subset B$ at each point of which the fibre is generic.}
  \label{figure::generic-positions}
\end{minipage}%
\begin{minipage}{.7\textwidth}
  \centering
  \def\svgwidth{.9\textwidth}
  \input{pdf_tex/single-valued-branch.pdf_tex}
  \caption{A piecewise correspondence between two copies of the single-valued branch at $t=3$ and $\Sigma_2$.}
  \label{figure::single-valued-branch}
\end{minipage}
\end{figure}

When $b\in \Gamma$ is approaching one of $-2$, $-1$, $0$, $1$ and $2$, there exist several closed curves on the generic fibre 
each joining two distinct branch points and vanishing when $b$ takes the limit.
These closed curves are called \emph{vanishing cycles}.
Figure \ref{figure::b-X-intersections}
tells us what the pair of branch points is for each vanishing cycle.
The peripheral monodromy is not the product of Dehn twists along vanishing cycles, but
the product of their squares.
One may compare its action on a transverse arc to the standard picture of a Dehn twist and a squared Dehn twist (see
Figure \ref{figure::branch_points_rotated}).

\begin{figure}[htb]
\centering
\begin{minipage}{.45\textwidth}
  \centering
  \def\svgwidth{1.0\textwidth}
  \input{pdf_tex/b-X-intersections.pdf_tex}
  \caption{Deformation of the branch points.}
  \label{figure::b-X-intersections}
\end{minipage}%
\hfill
\begin{minipage}{.45\textwidth}
  \centering
  \def\svgwidth{1.0\textwidth}
  \input{pdf_tex/branch_points_rotated.pdf_tex}
  \caption{A pair of branch points rotated clockwise and the squared Dehn twist.}
  \label{figure::branch_points_rotated}
\end{minipage}
\end{figure}

To identify the vanishing cycles, we deform the algebraic curve $\pi^{-1}(b)$ along $\Gamma$ and maintain the correspondence with $\Sigma_2$.
This deformation and all vanishing cycles at $-2$, $-1$, $0$, $1$, $2$ are depicted in Figure \ref{figure::peripheral-monodromy}.
We need the following:

\begin{figure}[htb]
\centering
\def\svgwidth{.9\textwidth}
\input{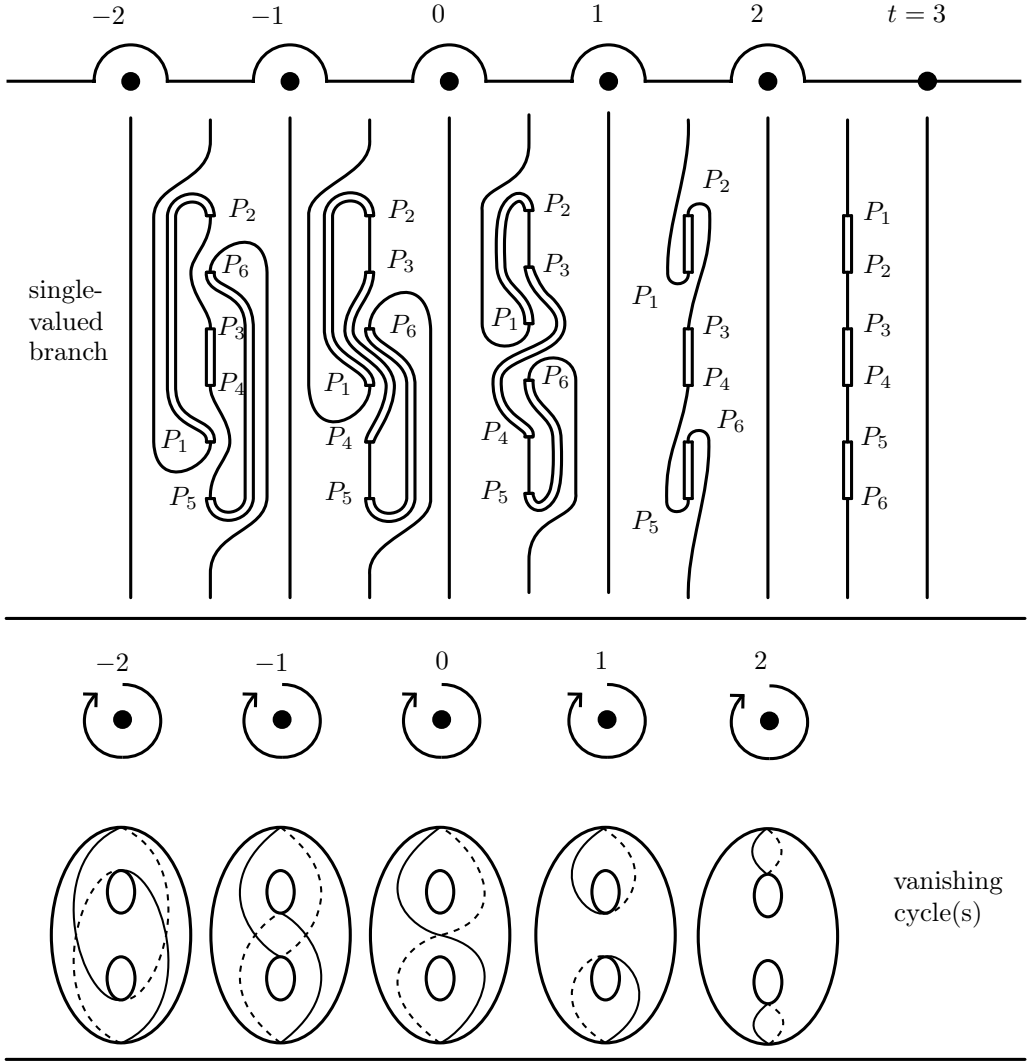}
\caption{Deformation of the generic fibre along $\Gamma$ is illustrated by the deformation of the single-valued branch, due to the correspondence between two copies of single-valued branch and $\Sigma_2$. One can further point out the vanishing cycle(s) at $2$, $1$, $0$, $-1$ and $-2$ respectively.}
\label{figure::peripheral-monodromy}
\end{figure}

\begin{lemma}
\label{lemma::triangle_operator_of_two_closed_curves}
Let $\gamma_1,\gamma_2\subset \Sigma_2$ be two closed curves on the (oriented) closed surface of genus $2$ such that $\Int(\gamma_1,\gamma_2)=1$.
Consider a path, denoted by $\gamma_1\triangle\gamma_2$, starting from some $p\in \gamma_1\setminus \gamma_2$, moving along $\gamma_1$ to the intersection,
turning right,
moving along $\gamma_2$ back to the intersection,
turning left and moving along $\gamma_1$ back to $p$.
Then, we have
\begin{equation*}
T_{\gamma_1\triangle\gamma_2}
=
T_{\gamma_2}^{-1}\circ T_{\gamma_1}\circ T_{\gamma_2}.
\end{equation*}
\end{lemma}

\begin{proof}[Proof of Lemma \ref{lemma::triangle_operator_of_two_closed_curves}]
This comes from the fact that $\gamma_1\triangle\gamma_2$
is homotopic to $T_{\gamma_2}^{-1}(\gamma_1)$.
\end{proof}

We return to the proof of Proposition \ref{proposition::example1}.
The monodromy homomorphism $F_*:\pi_1(B,t)\rightarrow\Mod_2$ is expressed as a sextuple,
denoted by $(\phi_{\infty},\phi_{-2},\phi_{-1},\phi_0,\phi_1,\phi_2)$,
where $\phi_{\infty}$, $\phi_{-2}$, $\phi_{-1}$, $\phi_0$, $\phi_1$ and $\phi_2$
are peripheral monodromies at $\infty$, $-2$, $-1$, $0$, $1$ and $2$.
Using $\phi_2^{1/2}$, $\phi_1^{1/2}$, $\phi_0^{1/2}$, $\phi_{-1}^{1/2}$ and $\phi_{-2}^{1/2}$ to denote half peripheral monodromies at $2$, $1$, $0$, $-1$ and $-2$ respectively, so that
\begin{align*}
\phi_2^{1/2}=&T_{\gamma_1}\circ T_{\gamma_5},\\
\phi_1^{1/2}=&T_{\gamma_1\triangle\gamma_2}\circ T_{\gamma_5\triangle\gamma_4}
= T_{\gamma_2}^{-1}\circ T_{\gamma_1} \circ T_{\gamma_2}\circ
T_{\gamma_4}^{-1}\circ T_{\gamma_5}\circ T_{\gamma_4},\\
\phi_0^{1/2}=&
T_{(\gamma_1\triangle \gamma_2)\triangle (\gamma_5\triangle \gamma_4\triangle \gamma_3)}
\\
=&
T_{\gamma_3}^{-1}\circ T_{\gamma_4}^{-1}\circ
T_{\gamma_5}^{-1}\circ T_{\gamma_4}\circ
T_{\gamma_3}\circ
T_{\gamma_2}^{-1} \circ T_{\gamma_1}\circ
T_{\gamma_2}\circ
T_{\gamma_3}^{-1}\circ T_{\gamma_4}^{-1}\circ
T_{\gamma_5}\circ T_{\gamma_4}\circ T_{\gamma_3},\\
\phi_{-1}^{1/2}=&
T_{\gamma_1\triangle\gamma_2\triangle\gamma_3}
\circ
T_{\gamma_5\triangle\gamma_4\triangle\gamma_3}
=
T_{\gamma_3}^{-1}\circ T_{\gamma_2}^{-1}\circ
T_{\gamma_1}\circ T_{\gamma_2}\circ T_{\gamma_3}
\circ
T_{\gamma_3}^{-1}\circ T_{\gamma_4}^{-1}\circ
T_{\gamma_5}\circ T_{\gamma_4}\circ T_{\gamma_3},
\\
\phi_{-2}^{1/2}=&
T_{\gamma_1\triangle\gamma_2\triangle\gamma_3\triangle\gamma_4}\circ
T_{\gamma_4\triangle\gamma_3\triangle\gamma_2\triangle\gamma_1}
\\
=&
T_{\gamma_4}^{-1}\circ T_{\gamma_3}^{-1}\circ
T_{\gamma_2}^{-1}\circ T_{\gamma_1}\circ
T_{\gamma_2}\circ T_{\gamma_3}\circ
T_{\gamma_4}\circ
T_{\gamma_2}^{-1}\circ T_{\gamma_3}^{-1}\circ
T_{\gamma_4}^{-1}\circ T_{\gamma_5}\circ
T_{\gamma_4}\circ T_{\gamma_3}\circ
T_{\gamma_2},
\end{align*}
we observe that
\begin{equation*}
T_{\tau}^2
\circ (\phi_{-2}^{1/2})^2
\circ (\phi_{-1}^{1/2})^2
\circ (\phi_0^{1/2})^2
\circ (\phi_1^{1/2})^2
\circ (\phi_2^{1/2})^2 = 1
\end{equation*}
and
\begin{equation*}
(\phi_2^{1/2})^2\circ (\phi_1^{1/2})^2\circ
(\phi_2^{1/2})^2\circ (\phi_1^{1/2})^2\circ
T_{\gamma_2}^4\circ T_{\gamma_4}^4\circ
T_{\sigma}^{-2} = 1
\end{equation*}
where $\tau$ and $\sigma$ are given in
Figure \ref{figure::genus2-2Dehntwists}.
As $\phi_\infty \circ \phi_{-2} \circ \phi_{-1} \circ \phi_{0} \circ \phi_{1} \circ \phi_{-2} = 1$,
the first equation shows that
the peripheral monodromy at $\infty$ is $\phi_\infty = T_{\tau}^2$, which is again a multi-twist.
The second equation shows that the monodromy along the essential closed curve $\beta_2 \cdot \beta_1 \cdot \beta_2 \cdot \beta_1$ is the multi-twist 
$T_{\sigma}^{2} \circ T_{\gamma_4}^{-4} \circ T_{\gamma_2}^{-4}$,
where the monodromy along $\beta_1$ is $\phi_1$ and the monodromy along $\beta_2$ is $\phi_2$.
Hence, the holomorphic family $C/B$ induces a quasi-isometrically but not isometrically immersed holomorphic curve.

The global monodromy $(\phi_{\infty},\phi_{-2},\phi_{-1},\phi_0,\phi_1,\phi_2)$ is a tuple in $\Mod_2$ whose components are of infinite order and pairwise intersecting.
Hence, the lift of the classifying map restricted to some fundamental polygon,
say $\left.\widetilde{F}\right|_{D}:D\rightarrow \Teich_2$,
is a quasi-isometric embedding.
\end{proof}

\subsection{Non quasi-isometrically embedded cusp regions}

In this subsection,
we provide a holomorphic curve of type $(0,8)$ in $\mathcal{M}_2$
for which a cusp region is not
quasi-isometrically embedded.
In fact, the corresponding peripheral monodromy is of finite order
and therefore this holomorphic curve does not satisfy the hypothesis of Theorem \ref{thmx::quasi-isometry} - (\RNum{1}).

\begin{example}
\label{example::non-quasi-cusp}
Set $B=\C \setminus \{-1,-1/2,-1/3,0,1/3,1/2,1\}$ and
\begin{equation*}
C=\left\{
\big(
[X_0:X_1],[Y_0:Y_1],b
\big)
\in \C P^1 \times \C P^1\times B
\middle|
\begin{aligned}
X_0^5 Y_1^2 =&
(X_1-(3b+2)X_0)
(X_1-bX_0)\\
&
(X_1+(3b-2)X_0)
(X_1+bX_0)\\
&
(X_1-X_0)Y_0^2
\end{aligned}
\right\}.
\end{equation*}
Let $\pi:C\rightarrow B$ be a holomorphic map with $\pi([X_0:X_1],[Y_0:Y_1],t)=t$.
Then, each fibre $\pi^{-1}(b)$ is a double cover of $\C P^1$ via
$\pi^{-1}(b) \ni ([X_0:X_1],[Y_0:Y_1],t)\mapsto [X_0:X_1]\in \C P^1$
with branch points
\begin{equation*}
P_1=[1:3b+2],
P_2=[1:b],
P_3=[1:1],
P_4=[1:-b],
P_5=[1:-3b+2]
\text{ and }
\infty=[0:1].
\end{equation*}
Therefore, $C/B$ is a holomorphic family of Riemann surfaces of genus $2$, say $\pi:C\rightarrow B$, which is non-isotrivial.    
\end{example}

\begin{proposition}
The classifying map $F:B\rightarrow \mathcal{M}_2$ of the holomorphic family $C/B$ in Example \ref{example::non-quasi-cusp} satisfies the following properties:
\begin{itemize}[-]
\item Peripheral monodromies at $\infty$ are of order $2$.

\item The restriction of $F$ to the cusp region at $\infty$ lies in a thick part of $\mathcal{M}_2$.
\end{itemize}
\end{proposition}

\begin{proof}
The base $B$ is a Riemann surface of type $(0,8)$.
To illustrate the monodromy homomorphism, we fix the base point $t\coloneqq 2\in B$ and investigate generic fibres at $b\in\Gamma\subset B$ where $\Gamma$ is given in Figure \ref{figure::example2-generic-positions}.

\begin{figure}[htb]
\centering
\def\svgwidth{.9\textwidth}
\input{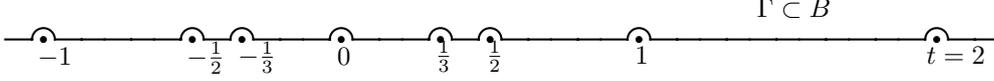}
\caption{The subset $\Gamma\subset B$ of generic positions we are looking at.}
\label{figure::example2-generic-positions}
\end{figure}

Figure \ref{figure::example2-single-valued-branch} shows a piecewise correspondence between $C_t\coloneqq\pi^{-1}(t)$ and the topological surface $\Sigma_2$.
Therefore, the monodromy homomorphism $F_*:\pi_1(B,t)\rightarrow \Mod_2$ is expressed as an octuple, denoted by
$(\phi_{\infty}, \phi_{-1}, \phi_{-1/2}, \phi_{-1/3}, \phi_0, \phi_{1/3}, \phi_{1/2}, \phi_1)$, where $\phi_{\infty}$, $\phi_{-1}$, $\phi_{-1/2}$, $\phi_{-1/3}$, $\phi_0$, $\phi_{1/3}$, $\phi_{1/2}$ and $\phi_1$ are peripheral monodromies at $\infty$, $-1$, $-1/2$, $-1/3$, $0$, $1/3$, $1/2$ and $1$.

\begin{figure}[htb]
\centering
\def\svgwidth{.63\textwidth}
\input{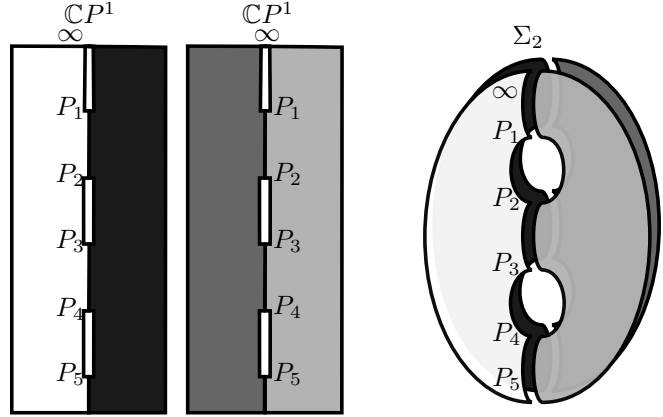}
\caption{A piecewise correspondence between $C_t$ and $\Sigma_2$.}
\label{figure::example2-single-valued-branch}
\end{figure}

The half peripheral monodromies are given by the following.
\begin{align*}
\phi_1^{1/2} &= T_{\gamma_3} \circ T_{\gamma_5},\\
\phi_{1/2}^{1/2} &=
T_{\gamma_3\triangle(\gamma_5\triangle\gamma_4)}
= T_{\gamma_4}^{-1} \circ T_{\gamma_5}^{-1} \circ T_{\gamma_4} \circ T_{\gamma_3}
\circ T_{\gamma_4}^{-1} \circ T_{\gamma_5} \circ T_{\gamma_4},\\
\phi_{1/3}^{1/3} &=
T_{\gamma_5\triangle \gamma_4}
= T_{\gamma_4}^{-1} \circ T_{\gamma_5} \circ T_{\gamma_4},\\
\phi_0^{1/2} &=
T_{\gamma_3\triangle\gamma_4} \circ T_{(\gamma_5\triangle\gamma_4)\triangle(\gamma_3\triangle\gamma_2)}\\
&= T_{\gamma_4}^{-1} \circ T_{\gamma_3} \circ T_{\gamma_4} \circ
T_{\gamma_2}^{-1} \circ T_{\gamma_3}^{-1} \circ T_{\gamma_2} \circ
T_{\gamma_4}^{-1} \circ T_{\gamma_5} \circ T_{\gamma_4} \circ
T_{\gamma_2}^{-1} \circ T_{\gamma_3} \circ T_{\gamma_2},\\
\phi_{-1/3}^{1/2} &=
T_{\gamma_3\triangle\gamma_2} = 
T_{\gamma_2}^{-1} \circ T_{\gamma_3} \circ T_{\gamma_2},\\
\phi_{-1/2}^{1/2} &=
T_{(\gamma_3\triangle\gamma_4)\triangle\gamma_2)} =
T_{\gamma_2}^{-1} \circ T_{\gamma_4}^{-1} \circ T_{\gamma_3} \circ T_{\gamma_4} \circ T_{\gamma_2},\\
\phi_{-1}^{1/2} &=
T_{\gamma_2} \circ T_{\gamma_4}.
\end{align*}
We compute that
$(\phi_{-1}^{1/2})^2 \circ
(\phi_{-1/2}^{1/2})^2 \circ
(\phi_{-1/3}^{1/2})^2 \circ
(\phi_{0}^{1/2})^2 \circ
(\phi_{1/3}^{1/2})^2 \circ
(\phi_{1/2}^{1/2})^2 \circ
(\phi_{1}^{1/2})^2$
is of order $2$.

\begin{remark}
By Theorem \ref{thmx::quasi-isometry} (or more directly, Proposition \ref{proposition::torsion-peripheral}), the finite order peripheral monodromy at $\infty$
forces $\left.F\right|_U$ to lie in a compact subset of $\mathcal{M}_2$.
The following provides an alternative computational verification.
\end{remark}

Fix an orientation preserving diffeomorphism $f_t:\Sigma_2\rightarrow C_t$ marking $C_t$ and endow each $C_b\coloneqq \pi^{-1}(b)$ with the marking $f_b$ along $\Gamma$, for $b\in \Gamma$.
We take a sufficiently large $N>0$.
To see that the restriction of $F$ to the cusp region at $\infty$ lies in a thick part of $\mathcal{M}_2$, it suffices to show that $\sys(C_b)$ is bounded away from $0$, for $b\in \R_{\ge N}$.
From now on, we consider only generic positions $b\in \R_{\ge N}$ and generic fibres $C_b$ at $b\in \R_{\ge N}$.

Half the hyperbolic distance on $C_b$ is equal to the Kobayashi distance on $C_b$.
Recall that the Kobayashi pseudo-norm on $T C_b$ is defined by $\Kob_{C_b}(x,v)=\inf_{\phi} \{1/c\}$
for $x\in C_b$ and $v\in T_x C_b$, where the infimum is taken over all holomorphic maps $\phi:\Delta \rightarrow C_b$ satisfying $\phi(0)=x$
and $(d\phi)_0 (\partial/\partial z)=c\cdot v$.
In order to obtain a very coarse estimation of $\Kob_{C_b}$,
we make the following restrictions on $\phi$: 
(\RNum{1}) $\phi(\Delta)$ lies in a single-valued branch, which is a subset of the affine chart $\C$;
(\RNum{2}) $\phi:\Delta\rightarrow\C$ is the composition of a linear map and a translation, meaning that $\phi(\Delta)\subset \C$ is also a disc away from $P_1,\ldots,P_5$.

The branch points $P_1,\ldots,P_5$ are colinear.
We observe that each ratio
\begin{equation*}
r_{i,j,k,l}(b)\coloneqq 
\frac{ d_{\R^2}(P_i,P_j) }{ d_{\R^2}(P_k,P_l) }
\end{equation*}
of euclidean distances converges as $b\rightarrow\infty$, for $i\neq j$ and $k\neq l$.
Set
\begin{equation*}
r_{\text{min}}\coloneqq (1/2)\min_{i,j,k,l}\{ \lim_{b\rightarrow\infty} r_{i,j,k,l}(b) \}
\text{~~~and~~~}
r_{\text{max}}\coloneqq (1/2)\max_{i,j,k,l}\{ \lim_{b\rightarrow\infty} r_{i,j,k,l}(b) \}.
\end{equation*}
Therefore, for each $i\in\{2,3,4,5\}$, there exists a closed curve $\beta_i$ lying in a single-valued branch and homotopic to $f_b(\gamma_i)$
such that
$l_{\R^2}(\beta_i)\le 12 r_{\text{max}}+\epsilon$
and
$d_{\R^2}(x,P_j)\ge r_{\text{min}}-\epsilon$ for all $x\in \beta_i$ and $j\in\{1,2,3,4,5\}$,
where $\epsilon>0$ is sufficiently small and determined by $N$ (see Figure \ref{figure::Kob}).

\begin{figure}[htb]
\centering
\def\svgwidth{.15\textwidth}
\input{pdf_tex/Kob_1.pdf_tex}
\hspace{1cm}
\def\svgwidth{.15\textwidth}
\begingroup%
  \makeatletter%
  \providecommand\color[2][]{%
    \errmessage{(Inkscape) Color is used for the text in Inkscape, but the package 'color.sty' is not loaded}%
    \renewcommand\color[2][]{}%
  }%
  \providecommand\transparent[1]{%
    \errmessage{(Inkscape) Transparency is used (non-zero) for the text in Inkscape, but the package 'transparent.sty' is not loaded}%
    \renewcommand\transparent[1]{}%
  }%
  \providecommand\rotatebox[2]{#2}%
  \newcommand*\fsize{\dimexpr\f@size pt\relax}%
  \newcommand*\lineheight[1]{\fontsize{\fsize}{#1\fsize}\selectfont}%
  \ifx\svgwidth\undefined%
    \setlength{\unitlength}{113.38582677bp}%
    \ifx\svgscale\undefined%
      \relax%
    \else%
      \setlength{\unitlength}{\unitlength * \real{\svgscale}}%
    \fi%
  \else%
    \setlength{\unitlength}{\svgwidth}%
  \fi%
  \global\let\svgwidth\undefined%
  \global\let\svgscale\undefined%
  \makeatother%
  \begin{picture}(1,4)%
    \lineheight{1}%
    \setlength\tabcolsep{0pt}%
    \put(0,0){\includegraphics[width=\unitlength,page=1]{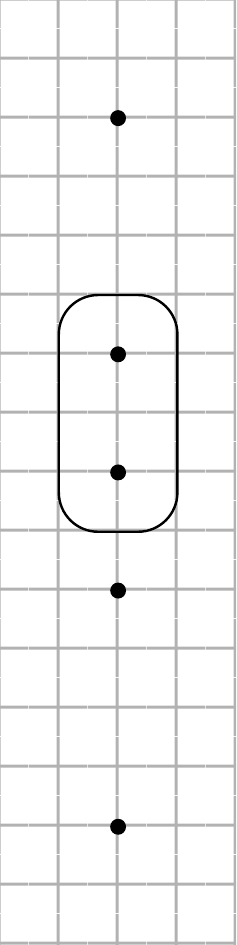}}%
    \put(0.55453305,3.41972213){\color[rgb]{0,0,0}\makebox(0,0)[lt]{\lineheight{1.25}\smash{\begin{tabular}[t]{l}$P_1$\end{tabular}}}}%
    \put(0.55453305,2.42284031){\color[rgb]{0,0,0}\makebox(0,0)[lt]{\lineheight{1.25}\smash{\begin{tabular}[t]{l}$P_2$\end{tabular}}}}%
    \put(0.55453305,1.41673355){\color[rgb]{0,0,0}\makebox(0,0)[lt]{\lineheight{1.25}\smash{\begin{tabular}[t]{l}$P_4$\end{tabular}}}}%
    \put(0.55453305,1.92479362){\color[rgb]{0,0,0}\makebox(0,0)[lt]{\lineheight{1.25}\smash{\begin{tabular}[t]{l}$P_3$\end{tabular}}}}%
    \put(0.55453305,0.42816353){\color[rgb]{0,0,0}\makebox(0,0)[lt]{\lineheight{1.25}\smash{\begin{tabular}[t]{l}$P_5$\end{tabular}}}}%
    \put(0.02086199,2.16494942){\color[rgb]{0,0,0}\makebox(0,0)[lt]{\lineheight{1.25}\smash{\begin{tabular}[t]{l}$\beta_3$\end{tabular}}}}%
    \put(0.75,3.75000005){\color[rgb]{0,0,0}\makebox(0,0)[lt]{\lineheight{1.25}\smash{\begin{tabular}[t]{l}$\C$\end{tabular}}}}%
  \end{picture}%
\endgroup%

\hspace{1cm}
\def\svgwidth{.15\textwidth}
\begingroup%
  \makeatletter%
  \providecommand\color[2][]{%
    \errmessage{(Inkscape) Color is used for the text in Inkscape, but the package 'color.sty' is not loaded}%
    \renewcommand\color[2][]{}%
  }%
  \providecommand\transparent[1]{%
    \errmessage{(Inkscape) Transparency is used (non-zero) for the text in Inkscape, but the package 'transparent.sty' is not loaded}%
    \renewcommand\transparent[1]{}%
  }%
  \providecommand\rotatebox[2]{#2}%
  \newcommand*\fsize{\dimexpr\f@size pt\relax}%
  \newcommand*\lineheight[1]{\fontsize{\fsize}{#1\fsize}\selectfont}%
  \ifx\svgwidth\undefined%
    \setlength{\unitlength}{113.38582677bp}%
    \ifx\svgscale\undefined%
      \relax%
    \else%
      \setlength{\unitlength}{\unitlength * \real{\svgscale}}%
    \fi%
  \else%
    \setlength{\unitlength}{\svgwidth}%
  \fi%
  \global\let\svgwidth\undefined%
  \global\let\svgscale\undefined%
  \makeatother%
  \begin{picture}(1,4)%
    \lineheight{1}%
    \setlength\tabcolsep{0pt}%
    \put(0,0){\includegraphics[width=\unitlength,page=1]{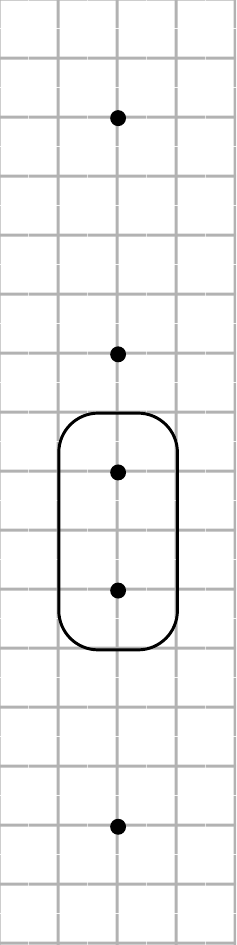}}%
    \put(0.55453305,3.41972213){\color[rgb]{0,0,0}\makebox(0,0)[lt]{\lineheight{1.25}\smash{\begin{tabular}[t]{l}$P_1$\end{tabular}}}}%
    \put(0.55453305,2.42284031){\color[rgb]{0,0,0}\makebox(0,0)[lt]{\lineheight{1.25}\smash{\begin{tabular}[t]{l}$P_2$\end{tabular}}}}%
    \put(0.55453305,1.41673355){\color[rgb]{0,0,0}\makebox(0,0)[lt]{\lineheight{1.25}\smash{\begin{tabular}[t]{l}$P_4$\end{tabular}}}}%
    \put(0.55453305,1.92479362){\color[rgb]{0,0,0}\makebox(0,0)[lt]{\lineheight{1.25}\smash{\begin{tabular}[t]{l}$P_3$\end{tabular}}}}%
    \put(0.55453305,0.42816353){\color[rgb]{0,0,0}\makebox(0,0)[lt]{\lineheight{1.25}\smash{\begin{tabular}[t]{l}$P_5$\end{tabular}}}}%
    \put(0.02086199,1.6622406){\color[rgb]{0,0,0}\makebox(0,0)[lt]{\lineheight{1.25}\smash{\begin{tabular}[t]{l}$\beta_4$\end{tabular}}}}%
    \put(0.75,3.75000005){\color[rgb]{0,0,0}\makebox(0,0)[lt]{\lineheight{1.25}\smash{\begin{tabular}[t]{l}$\C$\end{tabular}}}}%
  \end{picture}%
\endgroup%

\caption{Closed curves homotopic to $f_b(\gamma_2),\ldots,f_b(\gamma_5)$ with bounded Kobayashi lengths.}
\label{figure::Kob}
\end{figure}

Each of $\beta_2$, $\beta_3$, $\beta_4$ and $\beta_5$ has the hyperbolic length
\begin{align*}
l_{C_b}(\beta_i)
&= 2\cdot \int_0^1 \Kob_{C_b}(\beta_i(t),\dot{\beta}_i(t)) dt
\le 2\cdot \int_0^1 \frac{1}{\min_j\{d_{\R^2}(\beta_i(t),P_j)\} / |\dot{\beta}_i(t)|} dt
\\
&\le \frac{2}{r_{\min}-\epsilon} \int_0^1 |\dot{\beta}_i(t)| dt
= \frac{2 l_{\R^2}(\beta_i)}{r_{\min}-\epsilon}
\le \frac{2(12 r_{\text{max}}+\epsilon)}{r_{\text{min}}-\epsilon}.
\end{align*}
Hence $L_{\gamma_2}(C_b)$, 
$L_{\gamma_3}(C_b)$,
$L_{\gamma_4}(C_b)$ and
$L_{\gamma_5}(C_b)$
are uniformly bounded from above.
We conclude that $\sys(C_b)$ is bounded away from $0$.
\end{proof}

\subsection{Holomorphic genus-2 Lefschetz fibrations}
\label{subsection::genus2Lef}

Holomorphic genus-$2$ Lefschetz fibrations over a punctured sphere are quite well-understood by works of Siebert and Tian \cite{Siebert-Tian-2005-genus2} as well as Chakiris \cite{chakiris1983monodromy} and Smith \cite{Smith1999Lefschetz} (cf. also \cite{salepci2014lefschetz}).
In particular, there are only finitely many explicit possibilities for the global monodromies of genus-$2$ Lefschetz fibrations without reducible fibres (i.e. without separating vanishing cycles) up to Hurwitz moves.
By convention, we use $\bullet$ to denote the concatenation of tuples: $(\phi_1,\ldots,\phi_k)\bullet (\psi_1,\ldots,\psi_l) = (\phi_1,\ldots, \phi_k,\psi_1,\ldots,\psi_l)$.
The power of a tuple corresponds to a repeated concatenation with itself. The symbol $\prod$ represents
the concatenation of a family of tuples.

\begin{proposition}
\label{proposition::monodromy_genus2lef}
Given $n\ge 3$,
let $B$ be an oriented hyperbolic surface of type $(0,n)$
and
$F:B\rightarrow \mathcal{M}_2$ be a holomorphic map with a global monodromy $(\phi_1,\ldots,\phi_n)$.
Suppose each $\phi_i$ is the Dehn twist along a non-separating closed curve.
Then the global monodromy has the following properties:
\begin{enumerate}[(i)]
\item Using a finite sequence of Hurwitz moves, one can transform $(\phi_1,\ldots,\phi_n)$ into
the concatenation of tuples
$\mathcal{A}_1^p \bullet \mathcal{A}_2^q \bullet \mathcal{A}_3^r$
where $p,q,r$ are non-negative integers and
\begin{align*}
\mathcal{A}_1&=(T_{\gamma_1}, T_{\gamma_2}, T_{\gamma_3}, T_{\gamma_4}, T_{\gamma_5}, T_{\gamma_5}, T_{\gamma_4}, T_{\gamma_3}, T_{\gamma_2}, T_{\gamma_1})^2,\\
\mathcal{A}_2&=(T_{\gamma_1}, T_{\gamma_2}, T_{\gamma_3}, T_{\gamma_4})^5,~
\mathcal{A}_3=(T_{\gamma_1}, T_{\gamma_2}, T_{\gamma_3}, T_{\gamma_4}, T_{\gamma_5})^6.
\end{align*}

\item Using a finite sequence of Hurwitz moves, one can transform $(\phi_1,\ldots,\phi_n)$ into a tuple of Dehn twists along pairwise intersecting closed curves.
\end{enumerate}
\end{proposition}

Proposition \ref{proposition::monodromy_genus2lef} - (\RNum{1}) is Theorem B in \cite{chakiris1983monodromy}
and the holomorphic map $F:B\rightarrow \mathcal{M}_2$ is the classifying map of a holomorphic genus-$2$ Lefschetz fibration without separating vanishing cycles.
In this case, by Proposition \ref{proposition::monodromy_genus2lef} - (\RNum{2}), all hypotheses of Theorem \ref{thmx::quasi-isometry-fundamental-domain-n=0} hold for a specific fundamental polygon $D$
and we have the following corollary.

\begin{corollary}
Let $B=\Gamma\backslash \Hyperbolic^2$ be an oriented hyperbolic surface of type $(0,n)$, $n\ge 3$.
Let $F:B\rightarrow \mathcal{M}_2$ be the classifying map of a holomorphic genus-$2$ Lefschetz fibration without separating vanishing cycles.
Then, there exists a fundamental polygon $D$ of $B$ such that $\left.\widetilde{F}\right|_{D}: (D,(1/2)d_{\Hyperbolic^2}) \rightarrow (\Teich_2,d_{\Teich})$ is a $(2,K+\diam(D))$-quasi-isometric embedding, where $K=K(0,n,2,\sys(B))$ as in Theorem \ref{thmx::quasi-isometry-fundamental-domain-n=0}.
\end{corollary}

For positive integer $l$, recall that Hurwitz moves acting on 
a $l$-tuple $(\phi_1,\ldots,\phi_k)$ are given by
\begin{equation*}
(\ldots,\phi_{i}\circ\phi_{i+1}\circ\phi_i^{-1},\phi_i,\ldots)
\xmapsfrom{L_i}
(\ldots,\phi_i,\phi_{i+1},\ldots)
\xmapsto{R_i}
(\ldots,\phi_{i+1},\phi_{i+1}^{-1}\circ\phi_i\circ\phi_{i+1},\ldots).
\end{equation*}
We consider a closed curve $\delta\subset \Sigma_2$ and use $\vvec{\delta}$ to denote an orientation of $\delta$.
The \emph{algebraic intersection number} of two oriented closed curves 
$\vvec{\delta_1}$ and 
$\vvec{\delta_2}$,
denoted by $\hat{\Int}(\vvec{\delta_1},\vvec{\delta_2})$,
is defined as the sum of the indices of the intersection points of $\vvec{\delta_1}$ and $\vvec{\delta_2}$, where an intersection
point is of index $+1$ when the orientation of the intersection agrees with the
orientation of $\Sigma_{g,n}$ and is $-1$ otherwise.
Note that $\hat{\Int}(\vvec{\delta_1},\vvec{\delta_2})\neq 0$ only if $\Int(\delta_1,\delta_2)\neq 0$.

Let $\Omega_l$ be the set of $l$-tuples $(g_1,\ldots,g_l)$ where each $g_i$ is a positive Dehn twist in $\Mod_2$,
$\leftindex^{\sharp} \Omega_l$ be the set of $l$-tuples $(\vvec{\delta_1},\ldots,\vvec{\delta_l})$ where each $\vvec{\delta_i}$ is an orientation of some closed curve $\delta_i\subset \Sigma_2$.
There is a natural map $\natural: \leftindex^{\sharp}\Omega_l\rightarrow \Omega_l$ sending 
$(\vvec{\delta_1},\ldots,\vvec{\delta_l})$
to
$(T_{\delta_1},\ldots,T_{\delta_l})$.
We define the \emph{matrix of algebraic intersections} $\hat{M}=\hat{M}(\vvec{\delta_1},\ldots,\vvec{\delta_l}) \in \R^{l\times l}$ on every element in $\leftindex^{\sharp} \Omega_l$ by setting $\hat{M}_{i,j}=\hat{\Int}(\vvec{\delta_i},\vvec{\delta_j})$.

The maps $\leftindex^{\sharp} L_i$ and $\leftindex^{\sharp} R_i$ on $\leftindex^{\sharp} \Omega_l$
induced by the Hurwitz moves $L_i$ and $R_i$ are defined as follows.
\begin{equation*}
(\ldots, T_{\delta_i}(\vvec{\delta_{i+1}}), \vvec{\delta_i},\ldots)
\xmapsfrom{\leftindex^{\sharp} L_i}
(\ldots,\vvec{\delta_i},\vvec{\delta_{i+1}},\ldots)
\xmapsto{\leftindex^{\sharp} R_i}
(\ldots,\vvec{\delta_{i+1}},T_{\delta_{i+1}}^{-1}(\vvec{\delta_i}),\ldots).
\end{equation*}
We also have the maps
$\leftindex^{\flat} L_i$ and
$\leftindex^{\flat} R_i$
on $\R^{l\times l}$ defined by
\begin{equation*}
\hat{M}'=(m'_{j,k})
\xmapsfrom{\leftindex^{\flat} L_i}
\hat{M}=(m_{j,k}) 
\xmapsto{\leftindex^{\flat} R_i}
\hat{M}''=(m''_{j,k})
\end{equation*}
such that
$m'_{i,i}=m'_{i+1,i+1}=m''_{i,i}=m''_{i+1,i+1}=0$,
$m'_{i,i+1}=m_{i+1,i}=m''_{i,i+1}$,
$m'_{i+1,i}=m_{i,i+1}=m''_{i+1,i}$
and, for $j,k \not\in\{i,i+1\}$, that
$m'_{j,k}=m_{j,k}=m''_{j,k}$,
\begin{align*}
&m'_{i,k}=m_{i+1,k}+m_{i+1,i}m_{i,k},~
m'_{i+1,k}=m_{i,k}~,
m'_{j,i}=m_{j,i+1}-m_{i,i+1}m_{j,i},~
m'_{j,i+1}=m_{j,i},\\
&m''_{i,k}=m_{i+1,k},~
m''_{i+1,k}=m_{i,k}-m_{i,i+1}m_{i+1,k},~
m''_{j,i}=m_{j,i+1},~
m''_{j,i+1}=m_{j,i}+m_{i+1,i}m_{j,i+1}.
\end{align*}
Suppose that $q$ is a sequence of Hurwitz moves. We use $\leftindex^{\sharp}q$ to denote the sequence of corresponding maps on $\leftindex^{\sharp} \Omega_l$
and use $\leftindex^{\flat}q$ to denote the sequence of corresponding maps on $\R^{l\times l}$.

\begin{proposition}
\label{proposition::flat_of_Hurwitz_moves}
Let $(g_1,\ldots,g_l) \in \Omega_l$ be a tuple of positive Dehn twists in $\Mod_2$ where $g_i=T_{\delta_i}$ for $i=1,\ldots,l$.
Let $q$ be a sequence of Hurwitz moves.
Suppose that $(\vvec{\delta_1},\ldots,\vvec{\delta_l})\in \leftindex^{\sharp} \Omega_l$ is a lift of $(g_1,\ldots,g_l)$
and $(\vvec{\delta'_1},\ldots,\vvec{\delta'_l})\in \leftindex^{\sharp} \Omega_l$ is a lift of the resulting tuple $q\cdot (g_1,\ldots,g_l)$.
Then
\begin{equation*}
\leftindex^{\flat}q \cdot \hat{M} (\vvec{\delta_1},\ldots,\vvec{\delta_l}) =
\hat{M}(\vvec{\delta'_1},\ldots,\vvec{\delta'_l}).
\end{equation*}
\end{proposition}

\begin{proof}
It suffices to show that the following diagram is commutative.
\[\begin{tikzcd}
    {\Omega_l} & {\Omega_l} & {\Omega_l} \\
	{\leftindex^{\sharp} \Omega_l} & {\leftindex^{\sharp} \Omega_l} & {\leftindex^{\sharp} \Omega_l} \\
	{\R^{l\times l}} & {\R^{l\times l}} & {\R^{l\times l}}
	\arrow["{R_i}", from=1-2, to=1-3]
	\arrow["{\leftindex^{\sharp} R_i}", from=2-2, to=2-3]
	\arrow["{L_i}"', from=1-2, to=1-1]
	\arrow["{\leftindex^{\sharp} L_i}"', from=2-2, to=2-1]
	\arrow["{\leftindex^{\flat} R_i}", from=3-2, to=3-3]
	\arrow["{\leftindex^{\flat} L_i}"', from=3-2, to=3-1]
	\arrow["{\hat{M}}"', from=2-1, to=3-1]
	\arrow["{\hat{M}}"', from=2-2, to=3-2]
	\arrow["{\hat{M}}"', from=2-3, to=3-3]
	\arrow["\natural", from=2-1, to=1-1]
	\arrow["\natural", from=2-2, to=1-2]
	\arrow["\natural", from=2-3, to=1-3]
\end{tikzcd}\]
On the one hand, we have $T_{T_{\delta_i}(\delta_{i+1})}=T_{\delta_i}\circ T_{\delta_{i+1}}\circ T_{\delta_i}^{-1}$ and
$T_{T_{\delta_{i+1}}^{-1}(\delta_i)}=T_{\delta_{i+1}}^{-1}\circ T_{\delta_i} \circ T_{\delta_{i+1}}$ (see, e.g., \cite[Fact 3.7]{FarbMargalit2011}).
On the other hand, the algebraic intersection is well-defined on homology classes and we have
\begin{equation*}
[T_{\delta_i}(\vvec{\delta_{i+1}})] = [\vvec{\delta_{i+1}}] + \hat{\Int}(\vvec{\delta_{i+1}}, \vvec{\delta_{i}}) \cdot [\vvec{\delta_{i}}],
~~~
[T_{\delta_{i+1}}^{-1}(\vvec{\delta_i})] = [\vvec{\delta_i}] - \hat{\Int}(\vvec{\delta_i}, \vvec{\delta_{i+1}}) \cdot [\vvec{\delta_{i+1}}]
\end{equation*}
for $i=1,\ldots,l-1$ (see, e.g., \cite[Proposition 6.3]{FarbMargalit2011}).
\end{proof}

We prove Proposition \ref{proposition::monodromy_genus2lef} - (\RNum{2}) by starting with Lemma \ref{lemma::Hurwitz 1234554321} with a computer-assisted proof.

\begin{lemma}
\label{lemma::Hurwitz 1234554321}
For $i=1,2,3$,
there exists a finite sequence of Hurwitz moves $q_i$ such that the following holds:
Suppose that $\mathcal{A}_i$ is transformed by $q_i$ into a tuple of Dehn twists 
$(T_{\delta_{i,1}},\ldots,T_{\delta_{i,l_i}})$.
Then, for each pair of distinct curves in $\{\delta_{i,1},\ldots,\delta_{i,l_i}, \gamma_1\}$, the algebraic intersection is non-zero.
\end{lemma}

\begin{proof}
We consider the sequence of Hurwitz moves $q_i$ below and show that the resulting $l_i$-tuple $(T_{\delta_{i,1}},\ldots,T_{\delta_{i,l_i}})=q_i\cdot \mathcal{A}_i$ satisfies the desired properties.
Applying $q_i$ to the $(l_i+1)$-tuple $\mathcal{A}_i\bullet\{T_{\gamma_1}\}$, we obtain a $(l_i+1)$-tuple of the form $(T_{\delta_{i,1}},\ldots,T_{\delta_{i,l_i}}, T_{\gamma_1})$
since each component in $q_i$ is neither $L_{l_i}$ nor $R_{l_i}$.
Suppose that
$(\vvec{\delta_1}, \ldots, \vvec{\delta_{l_i+1}}) \in \leftindex^{\sharp} \Omega_{l_i}$ is the lift of $\mathcal{A}_i\bullet\{T_{\gamma_1}\}$ where the orientations of $\vvec{\gamma_1},\ldots,\vvec{\gamma_5}$ are shown in Figure \ref{figure::genus2-5Dehntwists}.
Consider the matrix of algebraic intersections $\hat{M}=\hat{M}(\vvec{\delta_1}, \ldots, \vvec{\delta_{l_i+1}})$ and apply $\leftindex^{\flat} q_i$ to $\hat{M}$.
By Proposition \ref{proposition::flat_of_Hurwitz_moves}, it suffices to verify that $\leftindex^{\flat}q_i \cdot \hat{M}$ has non-zero off-diagonal entries.
\begin{align*}
q_1 = (&
R_{3},  L_{10}, L_{2},  R_{16}, L_{17}, L_{1},  R_{8},  R_{16}, R_{11}, R_{12}, R_{4},  L_{13}, L_{11}, L_{15}, R_{10}, R_{13}, L_{5},  R_{12}, R_{9},  L_{8},  \\&
R_{2},  L_{6},  L_{18}, L_{7},  R_{6},  R_{19}, R_{4},  L_{8},  R_{5},  L_{9},  L_{9},  L_{14}, L_{8},  L_{8},  R_{18}, L_{10}, R_{13}, L_{6},  L_{6},  R_{19}, \\&
R_{7},  R_{12}, L_{8},  R_{7},  L_{4},  L_{11}, R_{3},  R_{2},  R_{12}, R_{2},  R_{13}, L_{10}, L_{5},  L_{4},  L_{10}, L_{17}, L_{18}, R_{11}, L_{12}, L_{13}, \\&
R_{6},  R_{4},  L_{10}, L_{12}, L_{13}, L_{14}, R_{4},  R_{15}, R_{1},  L_{12}, L_{16}, L_{17}, R_{16}, L_{18}, L_{14}, L_{18}, R_{17}, R_{1},  L_{14}, L_{14}, \\&
R_{3},  L_{5},  L_{5}
),\\
q_2 = (&
R_{8},  R_{15}, L_{2},  L_{6},  R_{10}, L_{13}, R_{16}, R_{5},  L_{17}, L_{8},  R_{14}, R_{3},  R_{11}, L_{10}, R_{6},  L_{12}, R_{8},  L_{4},  L_{7},  L_{9},  \\&
R_{3},  R_{6},  R_{14}, L_{12}, R_{11}, L_{9},  R_{16}, L_{1},  L_{2},  L_{17}, R_{7},  R_{13}, L_{8},  R_{14}, R_{14}, R_{1},  L_{15}, L_{9},  L_{18}, L_{14}, \\&
R_{10}, R_{13}, L_{1},  R_{16}, L_{15}, R_{1},  R_{7},  R_{10}, R_{10}, L_{19}, R_{16}, R_{15}, R_{3}
),\\
q_3 = (&
L_{12}, R_{18}, L_{7},  R_{24}, L_{14}, R_{7},  L_{10}, L_{27}, L_{5},  R_{1},  R_{19}, R_{10}, L_{24}, R_{20}, L_{28}, L_{27}, L_{11}, L_{22}, L_{8},  L_{10}, \\&
R_{9},  R_{24}, R_{7},  R_{2},  R_{19}, L_{10}, R_{25}, R_{26}, L_{25}, L_{8},  R_{1},  R_{10}, R_{20}, L_{21}, L_{20}, L_{25}, L_{3},  L_{10}, R_{23}, L_{9},  \\&
L_{7},  R_{6},  R_{10}, L_{13}, L_{8},  L_{27}, R_{24}, L_{21}, R_{22}, R_{12}, R_{6},  L_{11}, L_{5},  R_{14}, R_{8},  L_{11}, L_{18}, R_{7},  R_{4},  L_{12}, \\&
R_{15}, L_{20}, R_{7},  L_{11}, L_{12}, R_{10}, L_{19}, L_{10}, R_{3},  R_{17}, L_{14}, L_{23}, L_{9},  L_{25}, R_{24}, R_{4},  R_{22}, L_{13}, L_{18}, R_{3},  \\&
L_{29}, L_{24}, R_{22}, R_{22}, L_{16}, L_{18}, R_{28}, L_{29}, L_{29}, L_{6},  R_{1},  L_{24}, L_{14}, R_{1},  L_{20}, R_{17}, L_{8},  R_{25}, R_{12}, R_{14}, \\&
R_{8},  L_{1},  L_{20}, L_{22}, R_{7},  L_{18}, R_{14}, L_{21}, R_{7},  R_{18}, R_{9},  R_{14}, L_{2},  R_{1},  L_{19}, R_{11}, R_{27}, R_{9},  R_{4},  L_{26}, \\&
R_{20}, L_{10}, L_{23}, L_{22}, R_{20}, R_{17}, R_{5},  R_{5},  L_{8}
).
\end{align*}
\end{proof}
In the proof of Lemma \ref{lemma::Hurwitz 1234554321} the matrix $\leftindex^{\flat}q_i\cdot \hat{M}$ is hard to obtain manually
but can be quickly solved by a computer.
We implement these computations in $\text{Python}$ and make our code available on GitHub: \url{https://github.com/AHdoc/HurwitzMoves_to_AlgIntersections}.

\begin{proof}[Proof of Proposition \ref{proposition::monodromy_genus2lef} - (\RNum{2})]
By Proposition \ref{proposition::monodromy_genus2lef} - (\RNum{1}) and Lemma \ref{lemma::Hurwitz 1234554321}, using a sequence of Hurwitz moves, one can always transform a global monodromy into the concatenation of sub-tuples
\begin{equation*}
\mathcal{A}_u
\bullet
(T_{\delta_{1,1}},\ldots,T_{\delta_{1,l_1}})^{p'-1}
\bullet
(T_{\delta_{2,1}},\ldots,T_{\delta_{2,l_2}})^{q'-1}
\bullet
(T_{\delta_{3,1}},\ldots,T_{\delta_{3,l_3}})^{r'-1}
\end{equation*}
with some $u\in\{1,2,3\}$ and non-negative integers $p'$, $q'$, $r'$
such that $p'+q'+r'-3=p+q+r-1$, 
$\hat{\Int}(\gamma_1,\delta_{i,j})\neq 0$ for $i=1,2,3$, $j=1,\ldots,l_i$
and $\hat{\Int}(\delta_{i,j},\delta_{i,k})\neq 0$ for
$i=1,2,3$,
$1\le j\neq k\le l_i$.
By Lemma \ref{lemma::ivanov}, 
making $N$ sufficiently large such that
$N-2>\Int(\delta_{i,j},\delta_{i',j'})$
for all $i=1,2,3$, $i'=1,2,3$, $1\le j\le l_i$, $1\le j'\le l_{i'}$, we have
\begin{equation*}
\Int(T_{\gamma_1}^N(\delta_{i,j}),\delta_{i',j'})
\ge
(N-2) \Int(\gamma_1,\delta_{i,j})\Int(\gamma_1,\delta_{i',j'}) - \Int(\delta_{i,j},\delta_{i',j'}) \ge 1.
\end{equation*}
Since $T_{\delta_{i,1}}\cdots T_{\delta_{i,l_i}}$ is central for $i=1,2,3$,
one can further transform the global monodromy into
\begin{align*}
\mathcal{A}_u \bullet
\prod_{j=1}^{p'}
(T_{T_{\gamma_1}^{jN}(\delta_{1,1})},\ldots,T_{T_{\gamma_1}^{jN}(\delta_{1,l_1})})
&\bullet
\prod_{j=p'+1}^{p'+q'}
(T_{T_{\gamma_1}^{jN}(\delta_{2,1})},\ldots,T_{T_{\gamma_1}^{jN}(\delta_{2,l_2})})\\
&\bullet
\prod_{j=p'+q'+1}^{p'+q'+r'}
(T_{T_{\gamma_1}^{jN}(\delta_{3,1})},\ldots,T_{T_{\gamma_1}^{jN}(\delta_{3,l_3})}).
\end{align*}
Applying Lemma \ref{lemma::Hurwitz 1234554321} again,
we replace $\mathcal{A}_u$ with $(T_{\delta_{u,1}},\ldots,T_{\delta_{u,l_u}})$,
as desired.
\end{proof}

\sloppy
\printbibliography[
    heading=bibintoc,
    title={References}
]

\end{document}